\documentclass[reqno,12pt,letterpaper]{amsart}
\usepackage[proof]{sdmzmacros}

\newcommand{\RR}{{\mathbb R}}

\newcommand{\ZZ}{{\mathbb Z}}
\newcommand{\CI}{{C^\infty}}

\title{Microlocal analysis of forced waves}
\author{Semyon Dyatlov}
\email{dyatlov@math.berkeley.edu}
\address{Department of Mathematics, University of California, Berkeley, CA 94720}
\address{Department of Mathematics, MIT,
Cambridge, MA 02139}\author{Maciej Zworski}
\email{zworski@math.berkeley.edu}
\address{Department of Mathematics, University of California, Berkeley, CA 94720}

\begin{document}

\begin{abstract}
We use radial estimates for pseudodifferential operators to describe long time  evolution of solutions to $ i u_t - P u = f $ where $ P $ is a self-adjoint 0th order pseudodifferential operator satisfying hyperbolic dynamical assumptions
and where $ f $ is smooth.
This is motivated by recent results of Colin de Verdi\`ere and Saint-Raymond  \cite{SC} concerning a microlocal model of internal waves in stratified fluids.
\end{abstract}

\dedicatory{Dedicated to Richard Melrose on the occasion of his 70th birthday}

\maketitle

\section{Introduction}

Colin de Verdi\`ere and Saint-Raymond~\cite{SC} recently found an interesting connection between modeling of internal waves in stratified fluids and spectral theory of zeroth order pseudodifferential operators on compact 
manifolds. In other problems of fluid mechanics relevance of such operators
has been known for a long time, for instance in the work of Ralston \cite{Ra73}. We refer to \cite{SC} for pointers to current physics literature on internal waves and for numerical and experimental illustrations.

\begin{figure}
\includegraphics[width=7.5cm]{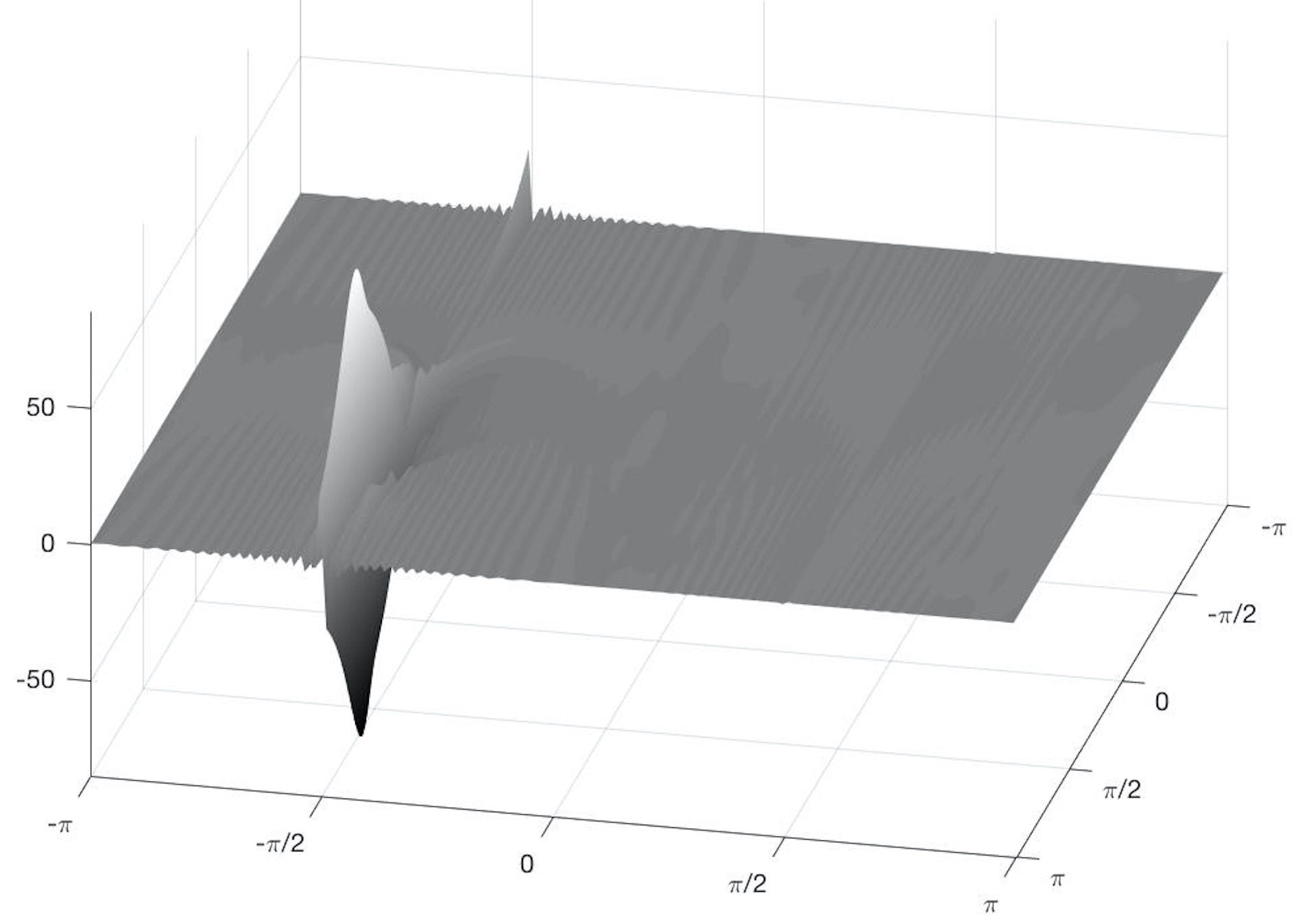}
\includegraphics[width=6.5cm]{flop.1}
\label{f:1}
\caption{On the left: the plot of the real part of $  u ( 50) $ for $ P =
\langle D \rangle^{-1} D_{x_2} + 2 \cos x_1 $ on $\mathbb T^2$ and $ f $ given by a smooth bump function centered at $ ( -\pi/2, 0 ) $.
We see the singularity formation on the line $ x_1 = -\pi/2 $.
On the right:  $ \Sigma :=
\kappa ( p^{-1} ( 0 ) ) \subset \partial \overline T^* \mathbb T^2$.
The
attracting Lagrangian, $ \Lambda^+_0 $,  comes from the highlighted circles.
See \S \ref{exa} for a discussion of the examples shown in the figures.
}
\end{figure}

The purpose of this note is to show how the main result of \cite{SC} (see 
also \cite{C}) follows from 
the now standard radial estimates for pseudodifferential operators. In particular, we avoid the use of Mourre theory, normal forms and Fourier integral operators and do not assume that the subprincipal symbols vanish. We also relax some geometric assumptions. The conclusions are formulated 
in terms of Lagrangian regularity in the sense of 
H\"ormander \cite[\S 25.1]{H3}. We illustrate the results with numerical examples. There are many possibilities for refinements but we restrict ourselves to applying off-the-shelf results at this stage.

Radial estimates were introduced by Melrose \cite{mel} for the study 
of asymptotically Euclidean scattering and have been developed further in 
various settings. We only mention some of the more relevant ones: scattering by zeroth order potentials (very close in spirit to the problems considered in \cite{SC}) by Hassell--Melrose--Vasy \cite{hmv}, asymptotically hyperbolic scattering by Vasy \cite{Va} (see also \cite[Chapter 5]{DZ} and \cite{V4D}) and
by Datchev--Dyatlov \cite{DaD}, in general relativity by Vasy \cite{Va}, Dyatlov \cite{dy} and Hintz--Vasy \cite{HiV}, and in hyperbolic dynamics by Dyatlov--Zworski \cite{DZ}. 
Particularly useful here is the work of Haber--Vasy \cite{hb} which generalized some of the results of \cite{hmv}. A very general version of radial estimates is presented ``textbook style" in \cite[\S E.4]{res}.

\subsection{The main result}

Motivated by internal waves in linearized fluids the authors of \cite{SC} considered long time behaviour of solutions to 
\begin{equation}
\label{eq:SC1}
\begin{gathered}
( i \partial_t - P ) u ( t ) =  f , \ \ u (0) =0, \ \
f \in \CI ( M ) , \\ 
P \in \Psi^0 (  M ) , \ \ P= P^* \end{gathered}
\end{equation}
where $M$ is a closed surface and
$ P $ satisfies dynamical assumptions presented in \S \ref{ass}. By changing
$ P $ to $ P - \omega_0 $ we can change $ f $ to the more physically relevant
oscillatory forcing term, $ e^{ - i \omega_0 t } f $.

Since the solution $ u ( t ) $ is given by 
\begin{equation}
\label{eq:uoft}  u ( t ) = -i\int_0^t e^{ - i s P  } f \, ds
= P^{-1} ( e^{ - it  P  }- 1)f 
,\end{equation}
(where the operator $ P^{-1} (e^{ -i t  P  }-1 )  $
is well defined for all $ t $ using the spectral theorem), the properties of the spectrum of~$P$ play a crucial role in the description of the long time behaviour of $ u ( t ) $. Referring to~\S\ref{ass} for the precise assumptions we state

\medskip
\noindent
{\bf Theorem.} {\em Suppose that the operator $ P $ satisfies assumptions~\eqref{eq:assP},  
\eqref{eq:dynaSC} below and that $ 0 \notin \Spec_{\rm{pp}} ( P ) $.  Then, for any $ f \in C^\infty ( M ) $, the solution to \eqref{eq:SC1} satisfies
\begin{equation}
\label{eq:SC2}
\begin{gathered}
 u ( t ) = u_\infty + b ( t ) + \epsilon ( t) ,\ \
\| b ( t ) \|_{L^2 } \leq C, \ \ \| \epsilon ( t ) \|_{ H^{-\frac12 - }}
\to 0 , \ \ t \to \infty  ,
\end{gathered}
\end{equation}
where  {(denoting by $H^{-\frac 12-}$ the intersection of the spaces $H^{-\frac 12-\varepsilon}$ over $\varepsilon>0$)}
\begin{equation}
\label{eq:DZ1}
u_\infty \in I^{0} ( M ; \Lambda^+_0 ) \subset H^{-\frac12-} ( M) 
\end{equation}
and $ I^{0} ( M ; \Lambda^+_0 ) $ is the space of Lagrangian distributions of order~$ 0 $ (see~\S\ref{s:lagrangian-basic}) associated 
to the attracting Lagrangian $ \Lambda^+_0 $ defined in \eqref{eq:Laplus}.}

The proof gives other results obtained in~\cite{SC}. In particular, we see that
in the neighbourhood of $ 0 $ the spectrum of~$ P$ is absolutely continuous except for finitely many eigenvalues with smooth eigenfunctions -- see~\S \ref{eig}.

In the case of  {general} Morse--Smale flows  {(allowing for fixed points)}, Colin de Verdi\`ere \cite[Theorem 4.3]{C} used a hybrid of 
Mourre estimates (in particular their finer version 
given by Jensen--Mourre--Perry \cite{jemp}) and of the radial estimates \cite[\S E.4]{res} to obtain a version of \eqref{eq:SC2} with an estimate on 
$ \WF ( u_\infty ) $. At this stage the purely microlocal approach of this paper would only give 
$ \| \epsilon ( t ) \|_{ H^{-\frac32 - }}
\to 0 $. 

\subsection{Assumptions on $ P $}
\label{ass}
We assume that $M$ is a compact surface without boundary
and $ P\in\Psi^0(M) $ is a 0th order pseudodifferential operator with 
principal symbol $ p \in S^0 ( T^* M \setminus 0;\mathbb R )  $ which is homogeneous (of order 0)
{and has $0$ as a regular value}. We also assume that for some smooth density, $ dm ( x ) $, on $M $, $ P $ is self-adjoint:
\begin{equation}
\label{eq:assP}
\begin{gathered}
P \in \Psi^0 ( M ) , \ \ P = P^* \text{ on $L^2 ( M , dm(x) ) $}, \\ p := \sigma (P ),  \ \ p ( x , t \xi ) = p ( x, \xi) , \ t > 0, \ \  dp |_{p^{-1} ( 0 ) } \neq 0. 
\end{gathered}
\end{equation}
The homogeneity assumption on $ p $ can be removed as the results of 
\cite[\S E.4]{res} and \cite{zazi} we use do not require it. That would however complicate the statement of the dynamical assumptions.

We use the notation of \cite[\S E.1.3]{res}, denoting by $ \overline T^* M $ the fiber-radially
compactified co-tangent bundle. Define the quotient map for the 
$ \RR^+ $ action, $ ( x, \xi ) \mapsto ( x , t \xi )$, $ t > 0 $, 
\begin{equation}
\label{eq:kap}  \kappa : \overline T^* M\setminus 0 \longrightarrow 
\partial \overline T^* M . 
\end{equation}
 {Denote by $|\xi|$ the norm of a covector $\xi\in T_x^*M$ with respect to some fixed Riemannian metric on~$M$.}
The rescaled Hamiltonian vector field $|\xi|H_p$ commutes with the $\mathbb R^+$ action and
\begin{equation}
\label{eq:defSig}
X:= \kappa_*(| \xi | H_p )\quad\text{is tangent to}\quad \Sigma := \kappa ( p^{-1} ( 0 ) ) . 
\end{equation}
Note that $\Sigma$ is an orientable surface since it is defined by the equation $p=0$
in the orientable 3-manifold $\partial \overline T^*M$.

We now recall the dynamical assumption made by Colin de Verdi\`ere and Saint-Raymond \cite{SC}:
\begin{equation}
\label{eq:dynaSC}
\text{ The flow of $ X $ on $ 
\Sigma  $ is a Morse--Smale flow with {\em no} fixed points. } 
\end{equation}
For the reader's convenience we recall the definition of Morse--Smale flows 
generated by $ X $ on a surface $ \Sigma $ (see \cite[Definition 5.1.1]{Surf}):
\begin{enumerate}
\item
 $ X $  has a finite number of fixed points all of which are hyperbolic;
\item
 $ X $ has a finite number of hyperbolic limit cycles;
\item
there are no separatrix connections between saddle fixed points; 
\item every trajectory different from (1) and (2) has unique trajectories
(1) or (2) as its $\alpha$, $\omega$-limit sets. 
\end{enumerate}
As stressed in \cite{SC}, Morse--Smale flows enjoy stability and genericity 
properties -- see \cite[Theorem 5.1.1]{Surf}. At this stage, following \cite{SC}, me make the strong assumption that there are no fixed points. 
By the Poincar\'e--Hopf Theorem that forces $ \Sigma $ to be a union of tori.

\begin{figure}
\includegraphics[width=7.5cm]{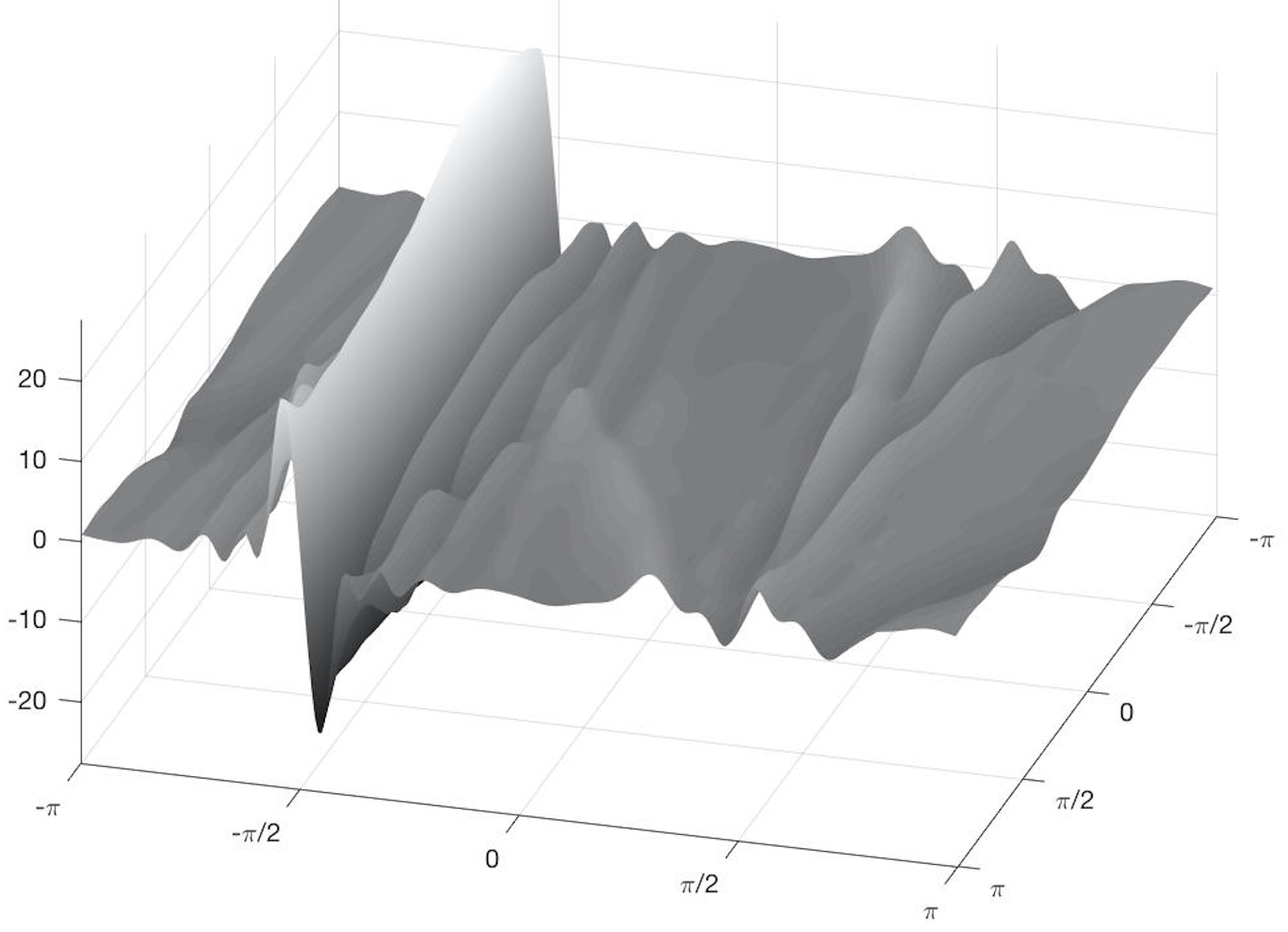}\includegraphics[width=6.5cm]{flop.2}
\caption{On the left: the plot of the real part of $  u ( 50) $ for 
$P$ given by~\eqref{eq:P2} and $ f $ given by a smooth bump function centered at $ ( -\pi/2, 0 ) $.
We see the singularity formation on the line $ x_1 = -\pi/2 $ and the slower
formation of singularity at $ x_1 = \pi/2 $.
On the right:  $ \Sigma :=
\kappa ( p^{-1} ( 0 ) )$. The 
attracting Lagrangian $ \Lambda^+_0 $  comes from the highlighted circles. }
\label{f:2}
\end{figure}

Under the assumption \eqref{eq:dynaSC}, the flow of~$X$ on~$ \Sigma $ has an attractor $ L^+_0$, which is a union of closed attracting curves. We define
the following {\em conic Lagrangian submanifold} of $ T^* M \setminus 0 $ (see \cite[\S 21.2]{H3}
and Lemma~\ref{l:sink-established}):
\begin{equation}
\label{eq:Laplus}  \Lambda^+_0 := \kappa^{-1} ( L^+_0 ) .
\end{equation}

\subsection{Examples}
\label{exa}

We illustrate the result with two simple examples
on $M:=\mathbb T^2=\mathbb S^1\times\mathbb S^1$
where $\mathbb S^1=\mathbb R/(2\pi\mathbb Z)$.
Denote $D:={1\over i}\partial$.
Consider first
\begin{equation}
\label{eq:P1}
\begin{gathered} 
P  := \langle D \rangle^{-1} D_{x_2} - 2 \cos x_1 , \ \ 
p = |\xi|^{-1} \xi_2 -2 \cos x_1 , \\
|\xi|H_p=-{\xi_1\xi_2\over |\xi|^2}\partial_{x_1}+{\xi_1^2\over|\xi|^2}\partial_{x_2}
-2(\sin x_1)|\xi|\partial_{\xi_1},\\
\Lambda^+_0 = \{ ( \pm \pi/2 , x_2; \xi_1 , 0 ) : x_2 \in \mathbb S^1 ,\ \pm \xi_1 < 0 \} .
\end{gathered}
 \end{equation}
In this case $ \kappa ( p^{-1} ( 0 ) ) $  (with $ \kappa $ given in \eqref{eq:kap}) is a union of two tori which do {\em not} cover~$ \mathbb T^2 $ (and thus does not satisfy the assumptions of~\cite{SC} but is covered by the treatment here, and  in \cite{C}). See Figure \ref{f:1} for the plot of $ \Re u ( t ) $, $ t = 50 $ and for a schematic visualization of $ \Sigma=\kappa (p^{-1} ( 0) ) $.

Our result applies also to the closely related operator 
\begin{equation}
\label{eq:P2}
\begin{gathered} 
P := \langle D \rangle^{-1} D_{x_2} - \tfrac12 \cos x_1 , \quad
p = |\xi|^{-1} \xi_2 - \tfrac 12 \cos x_1 ,\\
|\xi|H_p=-{\xi_1\xi_2\over |\xi|^2}\partial_{x_1}+{\xi_1^2\over|\xi|^2}\partial_{x_2}
-\tfrac12 {\sin x_1}|\xi|\partial_{\xi_1}.
\end{gathered}
 \end{equation}
The attracting Lagrangians are the same but the energy surface $ \kappa ( p^{-1} ( 0 ) )$ consists of two tori covering $ \mathbb T^2 $ (and hence satisfying the assumptions of~\cite{SC}) -- see Figure \ref{f:2}.

\section{Geometric structure of attracting Lagrangians}

In this section we prove geometric properties of the attracting and repulsive
Lagrangians for the flow $e^{t|\xi|H_p}$ where $p$ satisfies~\eqref{eq:dynaSC}.

\subsection{Sink and source structure}
\label{s:sink-source}

Let $ \Sigma ( \omega ) := \kappa ( p^{-1} ( \omega ) ) $. If $ \delta>0 $ is sufficiently small then 
stability of Morse--Smale flows (and the stability of non-vanishing of 
$ X $) shows that \eqref{eq:dynaSC} is satisfied for $ \Sigma ( \omega ) $, 
$|\omega|\leq 2\delta$. 
Let $ L^\pm_ \omega \subset \Sigma ( \omega) $ be the attractive ($+$) and repulsive ($-$) hyperbolic cycles for the flow of $ X $ on $ \Sigma ( \omega )  $. 
We first establish dynamical properties needed for the application of
radial estimates in \S \ref{reso}:
\begin{lemm}
  \label{l:sink-established}
$L^+ _ \omega$ is a radial sink and $ L^-_\omega  $ a radial source for the Hamiltonian flow 
of $ |\xi|(p - \omega)  = |\xi|\sigma ( P - \omega ) $ 
in the sense of \cite[Definition~E.50]{res}. The conic submanifolds
\[ \Lambda^\pm_\omega := \kappa^{-1} ( L^\pm_\omega )
\subset T^* M \setminus 0  \]
are Lagrangian.
\end{lemm}
\Remark It is not true that $L^\pm_\omega$ are radial sinks/sources
for the Hamiltonian flow of $p-\omega$ since~\cite[Definition~E.50]{res}
requires convergence of all nearby Hamiltonian trajectories,
not just those on the characteristic set $p^{-1}(\omega)$.
See Remark~3 following~\cite[Definition~E.50]{res} for details.
The singular behavior of $|\xi|$ at $\xi=0$ is irrelevant here
since we are considering a neighbourhood of the fiber infinity.
\begin{proof}
We consider the case of $L^+_\omega$ as that of  $L^-_\omega$ is similar.
To simplify the formulas below we put $\omega:=0$.
To see that $ \Lambda^+_0 $ is a Lagrangian submanifold we note that 
$ H_p $ and $ \xi \partial_\xi $ are tangent to $ \Lambda^+_0$
and independent (since $X$ does not vanish on $L^+_0$). Denoting the symplectic form by 
$ \sigma $, we have $ \sigma ( H_p , \xi \partial_\xi ) = -dp ( \xi \partial_\xi) = 0 $, that is $ \sigma $ vanishes on the tangent space to $ \Lambda^+_0 $.

We next show that $L^+_0$ is a radial sink. For simplicity assume that it consists of a single attractive closed trajectory of $X$ of period $T>0$, in particular $e^{TX}=I$ on~$L^+_0$.
Define the vector field
$$
Y:=H_{|\xi|p}
$$
which is homogeneous of order~0
on $T^*M\setminus 0$ and thus extends smoothly to the fiber-radial compactification
$\overline T^*M\setminus 0$, see~\cite[Proposition~E.5]{res}. We have
$Y=X$ on $\partial\overline T^*M\cap p^{-1}(0)$, thus $L^+_0\subset\partial\overline T^*M$
is a closed trajectory of $Y$ of period $T$.

Fix arbitrary $(x_0,\xi_0)\in L^+_0$ and define the linearized Poincar\'e map
$\mathcal P$ induced by $de^{TY}(x_0,\xi_0)$
on the quotient space $T_{(x_0,\xi_0)}(\overline T^*M)/\mathbb RY_{(x_0,\xi_0)}$.
The adjoint map $\mathcal P^*$
acts on covectors in $T^*_{(x_0,\xi_0)}(\overline T^*M)$ which annihilate $Y_{(x_0,\xi_0)}$.
To prove that $L^+_0$ is a radial sink it suffices to show that
the spectral radius of~$\mathcal P$ is strictly less than~1.

Put $\rho:=|\xi|^{-1}$
which is a boundary defining function on $\overline T^*M$,
then $\Sigma=\partial\overline T^*M\cap p^{-1}(0)$
is given by $\{p=0,\ \rho=0\}$.
Since $Y=X$ on $\Sigma$ and $L^+_0$ is an attractive cycle for~$X$ on~$\Sigma$, we have
$$
\mathcal P|_{\ker(dp)\cap\ker(d\rho)}=c_1\quad\text{for some }c_1\in\mathbb R,\ |c_1|<1.
$$
Since $Y$ is tangent to $\partial\overline T^*M=\rho^{-1}(0)$,
we have $Y\rho=f_2\rho$ for some $f_2\in C^\infty(\overline T^*M\setminus 0;\mathbb R)$.
Recalling that $Y=H_{|\xi|p}$ we compute
$Yp=pH_{|\xi|}p=-pH_p(\rho^{-1})=f_2p$. Denoting $c_2:=f_2(x_0,\xi_0)$ we then have
$$
\mathcal P^*(dp(x_0,\xi_0))=c_2 dp(x_0,\xi_0),\quad
\mathcal P^*(d\rho(x_0,\xi_0))=c_2 d\rho(x_0,\xi_0).
$$
Thus $\mathcal P$ has eigenvalues $c_1,c_2,c_2$. On the other hand,
$e^{TY}$ preserves the symplectic density $|\sigma\wedge\sigma|$ which has the form
$\rho^{-3} d\vol$ for some density $d\vol$ on $\overline T^*M$ which is smooth
up to the boundary. Taking the limit of this statement
at $(x_0,\xi_0)$ we obtain $\det\mathcal P=\det de^{TY}(x_0,\xi_0)=c_2^3$.
It follows that $c_1=c_2$ and thus $\mathcal P$ has spectral
radius $|c_1|<1$ as needed.
\end{proof}
For future use we define the conic hypersurfaces in $T^*M\setminus 0$
\begin{equation}
  \label{e:Lambda-pm-def}
\Lambda^\pm := \bigcup_{|\omega|<2\delta}\Lambda^\pm_\omega.
\end{equation}

\subsection{Geometry of Lagrangian families}
  \label{s:lagrangian-geometry}

We next establish some facts 
about families of Lagrangian submanifolds which do not need the dynamical assumptions~\eqref{eq:dynaSC}.
Instead we assume that:
\begin{itemize}
\item $p:T^*M\setminus 0\to\mathbb R$ is homogeneous of order~0;
\item $\Lambda\subset T^*M\setminus 0$ is a conic hypersurface;
\item $dp|_{T\Lambda}\neq 0$ everywhere;
\item the Hamiltonian vector field $H_p$ is tangent to $\Lambda$.
\end{itemize}
Under these assumptions, the sets
$$
\Lambda_\omega:=\Lambda\cap p^{-1}(\omega)
$$
are two-dimensional conic submanifolds of $T^*M\setminus 0$. Moreover, similarly
to Lemma~\ref{l:sink-established}, each $\Lambda_\omega$ is Lagrangian.
Indeed, if $G$ is a (local) defining function of $\Lambda$, namely
$G|_{\Lambda}=0$ and $dG|_{\Lambda}\neq 0$, then
$H_p$ being tangent to $\Lambda$ implies
\begin{equation}
  \label{e:p-G-commute}
\{p,G\}=0\quad\text{on}\quad \Lambda.
\end{equation}
Thus $H_p,H_G$ form a tangent frame on $\Lambda_\omega$ and $\sigma(H_p,H_G)=0$ on~$\Lambda$, where
$\sigma$ denotes the symplectic form.

Since $\xi\partial_\xi$ is tangent to each $\Lambda_\omega$, for any choice
of local defining function $G$ of $\Lambda$ we can write
\begin{equation}
  \label{e:Phi-new-def}
\xi\partial_\xi=\Phi H_p+\Theta H_G\quad\text{on}\quad\Lambda
\end{equation}
for some functions $\Phi,\Theta$ on $\Lambda$.
Since the one-dimensional subbundle
$\mathbb RH_G\subset T\Lambda$ is invariantly defined 
we see that $\Phi\in C^\infty(\Lambda;\mathbb R)$ does not depend on the choice of $G$.

The function $\Phi$ is homogeneous of order~1. Indeed, we can choose $G$ to be homogeneous of order~1 which implies
that $[\xi\partial_\xi,H_G]=0$; we also have $[\xi\partial_\xi,H_p]=-H_p$. 
By taking the commutator of both sides of~\eqref{e:Phi-new-def}
with $\xi\partial_\xi$ we see that $ \xi \partial_\xi \Phi = \Phi $.
 {Similarly we see that $\Theta$ is homogeneous of order~0.}

On the other hand, taking the commutators of both sides of~\eqref{e:Phi-new-def}
with $H_p$ and $H_G$ and using the following consequence of~\eqref{e:p-G-commute}, 
$$
[H_p,H_G]=H_{\{p,G\}}\in \mathbb RH_G\quad\text{on}\quad \Lambda,
$$ 
we get the following identities:
\begin{equation}
  \label{e:Phi-prop-1}
H_p\Phi\equiv 1,\quad
H_G\Phi\equiv 0\quad\text{on}\quad\Lambda.
\end{equation}
The function $\Phi$ is related to the $\omega$-derivative of a generating function of $\Lambda_\omega$
(see~\eqref{e:lm-par}):
\begin{lemm}
  \label{l:phase-der}
Assume that $\Lambda_\omega$ is locally given (in some coordinate system on~$M$) by
\begin{equation}
  \label{e:phase-der-1}
\Lambda_\omega=\{(x,\xi)\colon x=\partial_\xi F(\omega,\xi),\ \xi\in \Gamma_0\},
\end{equation}
where $ \xi \mapsto F ( \omega, \xi ) $ is a family of homogeneous functions of order~1 and $ \Gamma_0 \subset \RR^2 \setminus 0 $ is a cone. Then we have
\begin{equation}
  \label{e:phase-der-2}
\partial_\omega F(\omega,\xi)=-\Phi(\partial_\xi F(\omega,\xi),\xi).
\end{equation}
\end{lemm}
\begin{proof}
Let $G$ be a (local) defining function of $\Lambda$. Taking the $\partial_\xi$-component of~\eqref{e:Phi-new-def} 
at a point $\zeta:=(\partial_\xi F(\omega,\xi),\xi)\in\Lambda$ we have
\begin{equation}
  \label{e:clafouti-1}
\xi=-\Phi(\zeta)\partial_x p(\zeta)-\Theta(\zeta)\partial_x G(\zeta).
\end{equation}
On the other hand, differentiating in~$\omega$ the identities
$$
p(\partial_\xi F(\omega,\xi),\xi)=\omega,\quad
G(\partial_\xi F(\omega,\xi),\xi)=0
$$
we get
\begin{equation}
  \label{e:clafouti-2}
\langle \partial_x p(\zeta),\partial_\xi\partial_\omega F(\omega,\xi) \rangle=1,\quad
\langle \partial_x G(\zeta),\partial_\xi\partial_\omega F(\omega,\xi) \rangle=0.
\end{equation}
Combining~\eqref{e:clafouti-1} and~\eqref{e:clafouti-2} we arrive  {at}
$$
\langle\xi,\partial_\xi \partial_\omega F(\omega,\xi)\rangle=-\Phi(\zeta)
=-\Phi(\partial_\xi F(\omega,\xi),\xi)
$$
which implies~\eqref{e:phase-der-2} since the function $\xi\mapsto \partial_\omega F(\omega,\xi)$
is homogeneous of order~1.
\end{proof}
Now we specialize to the Lagrangian families used in this paper.
We start with a sign condition on $\Phi$  which will be used in \S \ref{asr}:
\begin{lemm}
  \label{l:Phi-sign}
Suppose that for $\Lambda=\Lambda^+$
or $\Lambda=\Lambda^-$, with $\Lambda^\pm$ given in~\eqref{e:Lambda-pm-def}
we define $ \Phi^\pm $ using~\eqref{e:Phi-new-def}.
Then for some constant $c>0$
\begin{equation}
  \label{e:Phi-sign}
\pm \Phi^\pm(x,\xi)\geq c|\xi|\quad\text{on}\quad \Lambda^\pm.
\end{equation}
\end{lemm}
\begin{proof}
We consider the case of $\Phi^+$ as the case of $\Phi^-$ is handled by
replacing $p$ with $-p$.
Recall from Lemma~\ref{l:sink-established}
that each $L^+_\omega=\kappa(\Lambda^+\cap p^{-1}(\omega))$ is a radial sink
for the flow $e^{t|\xi|H_p}$. Take $(x,\xi)\in \Lambda^+$ with $|\xi|$ large.
Then (with $S^*M$ denoting the cosphere bundle with respect to any fixed metric on~$M$)
\begin{equation}
  \label{e:lynmar}
e^{-tH_p}(x,\xi)\in S^*M\quad\text{for some}\quad t>0,\quad
t\sim |\xi|.
\end{equation}
Recall from~\eqref{e:Phi-prop-1} that $H_p\Phi^+=1$ on $\Lambda^+$. Thus
$$
\Phi^+(x,\xi)=\Phi^+(e^{-tH_p}(x,\xi))+t\geq  {c}|\xi|-C.
$$
It follows that $\Phi^+(x,\xi)\geq c|\xi|$ for large $|\xi|$; since $\Phi^+$ is homogeneous
of order~1, this inequality then holds on the entire $\Lambda^+$.
\end{proof}
We next construct adapted global defining functions of $\Lambda^\pm$ used in~\S\ref{s:lagreg}:
\begin{lemm}
  \label{l:G-construction}
Let $\Lambda^\pm$ be defined in~\eqref{e:Lambda-pm-def}. Then there exist 
$G_\pm\in C^\infty(T^*M\setminus 0;\mathbb R)$ such that:
\begin{enumerate}
\item $G_\pm$ are homogeneous of order~1;
\item $G_\pm|_{\Lambda^\pm}=0$ and $dG_\pm|_{\Lambda^\pm}\neq 0$;
\item $H_p G_\pm = a_\pm G_\pm$ in a neighborhood of $\Lambda^\pm$,
where $a_\pm\in C^\infty(T^*M\setminus 0;\mathbb R)$ are homogeneous
of order~$-1$ and $a_\pm|_{\Lambda^\pm}=0$.
\end{enumerate}
\end{lemm}
\begin{proof} We construct $G_+$, with $G_-$ constructed similarly.
Fix some function $\widetilde G_+$ which satisfies conditions~(1)~and~(2) of the
present lemma. It exists since $\Lambda^+$ is conic and orientable
(each of its connected components is diffeomorphic to~$[-\delta,\delta]\times \mathbb S^1\times\mathbb R^+$).
Let $\Theta_+$ be defined in~\eqref{e:Phi-new-def}:
\begin{equation}
  \label{e:gcon-1}
\xi\partial_\xi=\Phi_+H_p+\Theta_+H_{\widetilde G_+}\quad\text{on}\quad\Lambda^+.
\end{equation}
Commuting both sides of~\eqref{e:Phi-new-def} with $\xi\partial_\xi$
we see that $\Theta_+$ is homogeneous of order~0.
Moreover $\Theta_+$ does not vanish on $\Lambda^+$ since
$H_p$ is not radial (since the flow of $ X $ in \eqref{eq:defSig} has no fixed points). Choose $G_+$ satisfying conditions~(1)~and~(2) and such that
$$
G_+=\Theta_+\widetilde G_+\quad\text{near }\Lambda^+.
$$
Then~\eqref{e:gcon-1} gives
\begin{equation}
  \label{e:gcon-2}
\xi\partial_\xi=\Phi_+ H_p+H_{G_+}\quad\text{on}\quad\Lambda^+.
\end{equation}
We have $H_p G_+|_{\Lambda^+}=0$
 {(since $H_p$ is tangent to $\Lambda^+$)}, therefore $H_pG_+=a_+G_+$ near $\Lambda^+$
for some function $a_+$. Commuting both sides of~\eqref{e:gcon-2} with $H_p$
and using that $H_p\Phi_+\equiv 1$ on $\Lambda^+$ from~\eqref{e:Phi-prop-1} we have
$$
H_p=[H_p,\xi\partial_\xi]=H_p+[H_p,H_{G_+}]=H_p+H_{\{p,G_+\}}
=H_p+a_+H_{G_+}\quad\text{on}\quad\Lambda^+.
$$
Since $ H_{G_+} $ does not vanish on $ \Lambda^+ $, this gives $a_+|_{\Lambda^+}=0$
as needed.
\end{proof}
One application of Lemma~\ref{l:G-construction} is the existence of an $H_p$-invariant
density on $\Lambda^\pm$:
\begin{lemm}
  \label{l:density}
There exist densities $\nu^\pm_\omega$ on $\Lambda^\pm_\omega$, $\omega\in [-\delta,\delta]$, such that:
\begin{itemize}
\item $\nu^\pm_\omega$ are homogeneous of order~1, that is $\mathcal L_{\xi\partial_\xi}\nu^\pm_\omega=\nu^\pm_\omega$;
\item $\nu^\pm_\omega$ are invariant under $H_p$, that is $\mathcal L_{H_p}\nu^\pm_\omega=0$.
\end{itemize}
\end{lemm}
\begin{proof}
In the notation of Lemma~\ref{l:G-construction} define $\nu^\pm_\omega$ by
$
|\sigma\wedge \sigma|=|dp\wedge dG_\pm|\times \nu^\pm_\omega
$
where $\sigma$ is the symplectic form. The properties of $\nu^\pm_\omega$ follow from the identities
$$
\mathcal L_{\xi\partial_\xi}\sigma=\sigma,\quad
\mathcal L_{\xi\partial_\xi}dp=0,\quad
\mathcal L_{\xi\partial_\xi}dG_\pm=dG_\pm,\quad
\mathcal L_{H_p}\sigma=0
$$
and the following statement which holds on $\Lambda^\pm$:
$$
\hspace{1.6in} \mathcal L_{H_p}(dp\wedge dG_\pm)=dp\wedge d(a_\pm G_\pm)=0. \hspace{1.6in}\qedhere
$$
\end{proof}

\section{Resolvent estimates}
\label{reso}

Here we recall the radial estimates as presented in \cite[\S E.4]{res} 
specializing to the setting of \S \ref{ass}. We use the notation of~\cite[Appendix~E]{res}
and we write
$\|u\|_s:=\|u\|_{H^s(M)}$.

Since we are not in the semiclassical setting 
of \cite[\S E.4]{res} we will only use the usual notion of the wave front
set: for $ u \in \mathscr D' ( M ) $, $ \WF ( u ) \subset
T^* M \setminus 0 $ -- see \cite[Exercise~E.16]{res}.
Similarly, for
$A\in \Psi^k(M)$
we denote by $\Ell(A)\subset T^*M\setminus 0$ its (nonsemiclassical) elliptic set.
Both sets are conic.

\subsection{Radial estimates uniformly up to the real axis}
\label{rad} 

Since $ L^-_\omega  $ is a radial source we can apply 
\cite[Theorem~E.52]{res} (with $ h := 1 $) 
to the operator
$$
\widetilde P_\epsilon:=\widetilde P-i\epsilon\langle D\rangle\in\Psi^1(M),\quad
\widetilde P:=\langle D\rangle^{1/2} (P - \omega)\langle D\rangle^{1/2},\quad
 {0\leq\epsilon\ll 1}.
$$
Here, since $\widetilde P$ is self-adjoint, the threshold regularity condition
\cite[(E.4.39)]{res} is satisfied for $\widetilde P$ with any $ s >0 $.
Strictly speaking one has to modify the proof of~\cite[Theorem~E.52]{res}
to include the antiselfadjoint part $-i\epsilon \langle D\rangle$ which has a favorable
sign but is of the same differential order as $\widetilde P$.
(In~\cite{res} it was assumed that the principal symbol of $P$ is real-valued
near $L^-_\omega$.)
More precisely, we put $\mathbf P:=\widetilde P$ and $f:=\widetilde P_\epsilon u$
(instead of $f:=\widetilde P u$)
in~\cite[Theorem~E.52]{res}.
Since $\widetilde P_\epsilon$ satisfies the sign condition
for propagation of singularities~\cite[Theorem~E.47]{res},
it suffices to check that the positive commutator estimate~\cite[Lemma~E.49]{res}
holds. For that we write
\begin{equation}
  \label{e:crown}
\Im\langle f,G^*Gu\rangle_{L^2}=\Im\langle \widetilde Pu,G^*Gu\rangle_{L^2}-
\epsilon\Re\big\langle \langle D\rangle u,G^*Gu\big\rangle_{L^2}.
\end{equation}
Here $G\in\Psi^s(M)$ is the quantization of an escape function
used in the proof of~\cite[Lemma~E.49]{res};
recall that we put $h:=1$. We now estimate
the additional term in~\eqref{e:crown}:
$$
\begin{aligned}
-\Re\big\langle \langle D\rangle u,G^*Gu\big\rangle_{L^2}
&=-\|\langle D\rangle^{1/2}Gu\|_{L^2}^2+\langle \Re(G^*[\langle D\rangle,G])u,u\rangle_{L^2}\\
&\leq C\|B_1u\|_{s-1/2}^2+C\|u\|_{H^{-N}}^2
\end{aligned}
$$
where  {$B_1$ satisfies the properties in the statement of~\cite[Lemma~E.49]{res}} and in the last line we used that
$G^*[\langle D\rangle,G]\in\Psi^{2s}(M)$ has purely imaginary principal symbol
and thus $\Re(G^*[\langle D\rangle,G])\in \Psi^{2s-1}(M)$.
The rest of the proof of~\cite[Lemma~E.49]{res} applies without changes.
See also~\cite[Lemma~3.7]{DG}.

Applying the radial estimate in~\cite[Theorem~E.52]{res} for the operator
$\widetilde P_\epsilon=\langle D\rangle^{1/2}(P-\omega-i\epsilon)\langle D\rangle^{1/2}$ to $\langle D\rangle^{-1/2}u$
we see that for every $ \widetilde B_- \in \Psi^0(M) $, $ 
\Lambda^- \subset \Ell ( \widetilde B_- ) $ there exists
$ A_- \in \Psi^0 ( M ) $, $  \Lambda^-  
\subset \Ell( A_- ) $, such that
\begin{equation}
\label{eq:rad_source}
\begin{gathered} 
\| A_- u \|_{ s } \leq C \|\widetilde B_-( P - \omega - i \epsilon ) u \|_{ s+1} + 
C \| u \|_{ -N } , \\
u \in C^\infty ( M ) , \quad s > -\tfrac12 ,\quad
|\omega|\leq\delta,\quad
\epsilon \geq 0 , 
\end{gathered}
\end{equation}
where $ C $ does not depend on $ \epsilon, \omega $ and $ N$ can be chosen arbitrarily large.
The supports of $ A_-$, $\widetilde B_-$ are shown on  Figure~\ref{f:3}.

The inequality~\eqref{eq:rad_source} can be 
extended to a larger class of distributions
 {(as opposed to $u\in C^\infty(M)$)}: it suffices that  $  \widetilde B_- (P -\omega - i \epsilon )u \in H^{ s + 1} ( M ) $ and that 
$ A_- u \in H^{s'} ( M ) $ for some $ s' > -\frac12 $.
See Remark~5 after \cite[Theorem~E.52]{res} or
\cite[Proposition~2.6]{DZ}, \cite[Proposition~2.3]{Va}.
\begin{figure}
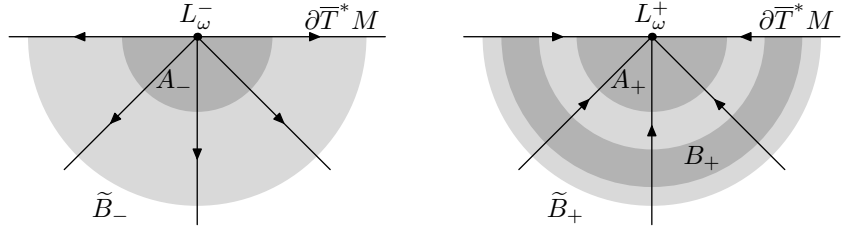

\includegraphics{flop.3}
\qquad
\includegraphics{flop.4}
\caption{An illustration of the supports of the operators
appearing in \eqref{eq:rad_source} (left: radial sources)
and \eqref{eq:rad_sink} (right: radial sinks). The horizontal line
on the top denotes $\partial\overline T^*M$, the arrows denote
flow lines of $|\xi|H_p$.}
\label{f:3}
\end{figure}

Similarly we have estimates near radial sinks~\cite[Theorem~E.54]{res}
for $L^+_\omega$.
Namely, for every $ \widetilde B_+ \in \Psi^0 ( M ) $, 
$ \Lambda^+ \subset \Ell (\widetilde B_+ ) $, there exist $ A_+,B_+ \in \Psi^0(M) $, such that $  \Lambda^+  
\subset {\rm{ell}}( A_+ ) $,
$ \WF ( B_+ ) \cap \Lambda^+ =\emptyset $, and 
\begin{equation}
\label{eq:rad_sink}
\begin{gathered}
\| A_+ u \|_{ s } \leq C \| \widetilde B_+ ( P - \omega - i \epsilon ) u \|_{ s+1} + 
C \| B_+ u \|_{ s } + 
C \| u \|_{ -N } , \\
u \in C^\infty ( M ) , \quad s < -\tfrac12 ,  \quad
|\omega|\leq\delta,\quad
\epsilon \geq 0 , 
\end{gathered}
\end{equation}
where $ C $ does not depend on $ \epsilon,\omega$ and $ N$ can be chosen arbitrarily large. 
The inequality is also valid for distributions $ u $ such that 
$ \widetilde B_+ ( P - \omega - i \epsilon ) u \in H^{s+1}( M )  $ and $ B_+ u \in H^{s} ( M)  $ and it then provides (unconditionally) $ A_+ u \in H^s ( M ) $~--
see Remark~2 after~\cite[Theorem~E.54]{res} or
\cite[Proposition~2.7]{DZ}, \cite[Proposition~2.4]{Va}.

Away from radial points we have the now standard propagation results of
Duistermaat--H\"ormander \cite[Theorem~E.47]{res}: if 
$ A, B, \widetilde B \in \Psi^0 ( M ) $ and 
for each $(x,\xi)\in \WF(A)$ there exists
$T\geq 0$ such that 
\[ e^{-T|\xi|H_p}(x,\xi)\in\Ell(B), \ \  e^{-t |\xi| H_p } ( x, \xi) \in 
\Ell ( \widetilde B ) , \ 0 \leq t \leq T , \]
then 
\begin{equation}
\label{eq:DH}
\begin{gathered}
\| A u \|_{ s } \leq C \| \widetilde B ( P - \omega - i \epsilon ) u \|_{ s+1} + 
C \| B u \|_{ s } + 
C \| u \|_{ -N } , \\
u \in C^\infty ( M ) ,  \quad s \in \RR , \quad |\omega|\leq\delta , \quad
\epsilon \geq 0 , 
\end{gathered}
\end{equation}
with $ C $ independent of $ \epsilon,\omega $.
{We also have the elliptic estimate~\cite[Theorem~E.33]{res}:
\eqref{eq:DH} holds with $B=0$ if $\WF(A)\cap p^{-1}([-\delta,\delta])=\emptyset$
and $\WF(A)\subset\Ell(\widetilde B)$.}

Let us now consider 
\[
 u_\epsilon = u_\epsilon ( \omega )  := ( P - \omega - i \epsilon )^{-1} f , \ \ 
f \in C^\infty ( M ) , \quad|\omega|\leq\delta,\quad \epsilon > 0 .
\] 
For any fixed $ \epsilon > 0 $, $ P - \omega - i \epsilon \in \Psi^0(M) $ is an elliptic operator
(its principal symbol equals $p-\omega-i\epsilon$ and $p$ is real-valued), thus by elliptic regularity $ u_\epsilon \in C^\infty ( M ) $. Combining 
\eqref{eq:rad_source}, \eqref{eq:rad_sink} and \eqref{eq:DH} we see that 
for any $ \beta > 0 $
\begin{equation}
\label{eq:uep1}  
\| u_\epsilon \|_{ - \frac12 - \beta } \leq C \| f \|_{ \frac12 + \beta}  + C \| u_\epsilon \|_{ -N}  , 
\end{equation}
and that 
\begin{equation}
\label{eq:uep2} \| A u_\epsilon \|_{ s } \leq C \| f \|_{s+1} + C \| u_\epsilon \|_{ -N } , \ \ 
\WF ( A ) \cap \Lambda^+  = \emptyset, \ \ s > - \tfrac12 . 
\end{equation}
Here the constant $C$ depends on $\beta,s$ but does not depend on $\epsilon,\omega$.
Indeed, by our dynamical assumption~\eqref{eq:dynaSC}
every trajectory $e^{t|\xi|H_p}(x,\xi)$ with $(x,\xi)\in p^{-1}([-\delta,\delta])\setminus\Lambda^+$ 
converges to $\Lambda^-$ as $t\to -\infty$ (see Figure~\ref{f:global}).
Applying~\eqref{eq:DH} with $B:=A_-$ and using~\eqref{eq:rad_source} we get~\eqref{eq:uep2}.
Putting $A:=B_+$ in~\eqref{eq:uep2} and using~\eqref{eq:rad_sink} we get~\eqref{eq:uep1}. 
\begin{figure}
\includegraphics{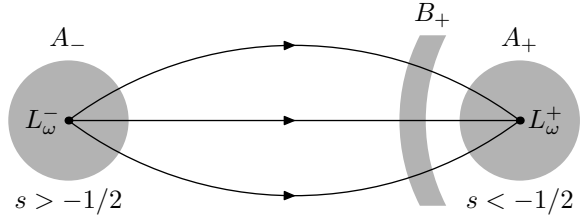}
\caption{A schematic representation of the flow
$e^{t|\xi|H_p}$ on the fiber infinity $\partial\overline T^*M$
intersected with the energy surface $p^{-1}(\omega)$,
with the regularity thresholds for the estimates~\eqref{eq:rad_source}
and~\eqref{eq:rad_sink}.}
\label{f:global}
\end{figure}

In particular, we obtain a regularity statement for the limits of the family $(u_\epsilon)$:
\begin{equation}
\label{eq:uesj}   \exists \, \epsilon_j \to 0 , \ u \in \mathscr D' ( M ) , \
 u_{\epsilon_j } \xrightarrow{ \mathscr D' ( M)  }  u  \quad
\Longrightarrow \quad u \in H^{ -\frac12 - } ( M ) , \ \ 
\WF ( u ) \subset \Lambda^+  .
\end{equation}
Note also that every $u$ in~\eqref{eq:uesj} solves the equation $(P-\omega)u=f$.

\subsection{Regularity of eigenfunctions}
\label{eig}
 
Motivated by \eqref{eq:uesj} we have the following regularity statement. 
The proof is an immediate modification of the proof of 
\cite[Lemma 2.3]{zazi}: replace $ P $ there by 
$ A^{-1} ( P - \omega ) A^{-1} $ where $ A \in  \Psi^{-\frac12} 
( M ) $ is elliptic, self-adjoint on $ L^2 (M , dm ( x ) ) $ 
(same density with respect to which $ P $ is self-adjoint) and invertible. 
We record this as
\begin{lemm}
\label{l:zazi}
Suppose that $P $ satisfies \eqref{eq:assP} and \eqref{eq:dynaSC}. Then for 
$ \omega $ sufficiently small and for $ u \in \mathscr D' ( M ) $
\[
( P - \omega ) u \in C^\infty , \quad
\WF ( u ) \subset \Lambda^+  , \quad
\Im \langle ( P - \omega ) u , u \rangle \geq 0 ,\quad
|\omega|\leq\delta
\]
implies that $ u \in C^\infty ( M ) $. 
\end{lemm}
In particular this shows that if $ ( P - \omega ) u = 0 $ and
$ \WF ( u) \subset \Lambda^+  $ then $ u \in L^2 $, that is 
$ \omega$ lies in the point spectrum $\Spec_{\rm{pp}} ( P ) $. Radial estimates then show that
the number of such $ \omega$'s is finite in a neighbourhood of $ 0 $:
\begin{lemm}
\label{l:spec}
Under the assumptions \eqref{eq:assP} and \eqref{eq:dynaSC}, with $ \delta $ sufficiently small, 
\begin{equation}
\label{eq:spec}
\begin{gathered}
| \Spec_{\rm{pp}} ( P ) \cap [- \delta, \delta ] | < \infty ; \\ 
( P - \omega ) u  = 0 , \ u \in L^2 ( M ) , \ 
|\omega | \leq \delta \quad \Longrightarrow\quad u \in C^\infty ( M ) .
\end{gathered}
\end{equation}
\end{lemm}
\begin{proof}
If $ u \in L^2 ( M ) $ then the threshold assumption in 
\eqref{eq:rad_source} is satisfied for $ P - \omega $ near $ \Lambda^-  $ and for $ - ( P - \omega ) $ near $ \Lambda^+  $.
Using the remark about regularity after \eqref{eq:rad_source},
as well as~\eqref{eq:DH} away from sinks and sources,
we conclude that 
\begin{equation}
\label{eq:N2s} \| u \|_{ s} \leq C \| u \|_{ -N }
\end{equation}
for any $ s $ and $ N $. That implies that $ u \in C^\infty ( M ) $.
Now, suppose that there exists an infinite set of $ L^2 $ eigenfunctions with 
eigenvalues in $ [ - \delta, \delta ] $:
\[
( P - \omega_j ) u_j = 0 , \ \ \ \langle u_k, u_j \rangle_{ L^2 ( M) }
= \delta_{kj} , \ \ \ | \omega_j | \leq \delta.
\]
Since $ u_j \rightharpoonup 0 $, weakly in $ L^2 $, $ u_j \to 0 $ strongly in 
$ H^{-1} $. But this contradicts \eqref{eq:N2s} applied with $ s = 0 $ and 
$ N = 1 $.
\end{proof}
{From now on we make the assumption that $P$ has no eigenvalues in $[-\delta,\delta]$:
\begin{equation}
  \label{e:no-spectrum}
\Spec_{\rm{pp}} ( P ) \cap [- \delta, \delta ]=\emptyset.
\end{equation}
By Lemma~\ref{l:spec} we see that~\eqref{e:no-spectrum} holds for $\delta$ small enough
as long as $0\notin\Spec_{\rm{pp}}(P)$.}

\subsection{Limiting absorption principle}
\label{lap}

Using results of \S\S\ref{rad},\ref{eig} we obtain a version of the limiting absorption principle sufficient for proving~\eqref{eq:SC2}. Radial estimates 
can also easily give existence of $ ( P - \omega - i 0)^{-1} :
H^{\frac12+} ( M ) \to H^{-\frac12 - } ( M ) $ but we restrict ourselves to the simpler version and follow Melrose \cite[\S 14]{mel}. 
The only modification lies in replacing scattering asymptotics by the regularity result given in Lemma \ref{l:zazi}.
\begin{lemm}
\label{l:lap}
Suppose that $ P $ satisfies \eqref{eq:assP}, \eqref{eq:dynaSC}, and~\eqref{e:no-spectrum}.
Then for $ |\omega | \leq \delta $ and $ f \in C^\infty ( M ) $, the limit
\[
( P - \omega - i \epsilon )^{-1}f 
\xrightarrow{ H^{-\frac12 -  } ( M)  } ( P - \omega - i 0 )^{-1} f,\quad
\epsilon\to 0+
\]
exists. This limit is the unique solution to the equation
\begin{equation}
  \label{e:lapidus}
(P-\omega)u=f,\quad \WF(u)\subset\Lambda^+,
\end{equation}
and the map $\omega\mapsto (P-\omega-i0)^{-1}f\in H^{-\frac 12-}(M)$
is continuous in $\omega\in [-\delta,\delta]$.
\end{lemm}
\Remark Replacing $P$ with $-P$ we see that there is also a limit
$$
( P - \omega + i \epsilon )^{-1}f 
\xrightarrow{ H^{-\frac12 -  } ( M)  } ( P - \omega + i 0 )^{-1} f,\quad
\epsilon\to 0+
$$
which satisfies~\eqref{e:lapidus} with $\Lambda^+$ replaced by $\Lambda^-$.
\begin{proof} 
We first note that Lemma~\ref{l:zazi} and the spectral assumption~\eqref{e:no-spectrum}
imply that~\eqref{e:lapidus} has no more than one solution. 
By~\eqref{eq:uesj}, if a (distributional) limit 
$ ( P  - \omega-i \epsilon_j )^{-1} f $, $ \epsilon_j \to 0 $, exists then it solves~\eqref{e:lapidus}.

To show that the limit exists put $ u_\epsilon := ( P - \omega- i \epsilon )^{-1} f $ and suppose first that $ \| u_\epsilon \|_{ -\frac12 - \alpha }$
is not bounded as $ \epsilon \to 0 + $ for some $ \alpha > 0 $. Hence there exists $ \epsilon_j \to 0+ $ such that $ \| u_{\epsilon_j} \|_{ -\frac12 - \alpha } \to \infty $.
Putting
$ v_j := u_{\epsilon_j} / \| u_{\epsilon_j}  \|_{ -\frac12 - \alpha } $ we obtain
\begin{equation}
\label{eq:Pie} 
( P - \omega-i \epsilon_j ) v_j = f_j , \ \ \| v_j \|_{ {-\frac12 - \alpha }} = 1, 
\ \ f_j \xrightarrow{ C^\infty ( M)  } 0 . 
\end{equation}
Applying \eqref{eq:uep1} with  $ N = \frac12 + \alpha $
we see that $ v_j $ is bounded in $ H^{-\frac12 - \beta } ( M ) $ for any $ \beta > 0 $. 
Since $ H^{-\frac12 - \beta } ( M ) \hookrightarrow H^{-\frac12 - \alpha } ( M ) $, $ \beta < \alpha $ is compact we can assume, by passing to a subsequence,  that $ v_j \to v $ in 
$ H^{-\frac12 - \alpha } ( M ) $. Then 
$ (P-\omega) v = 0 $ and the same reasoning that led to \eqref{eq:uesj} shows that
$ \WF ( v ) \subset \Lambda^+ $. Thus $v$ solves~\eqref{e:lapidus} with $f\equiv 0$,
implying that $ v \equiv 0 $. This gives a contradiction with the normalization $\|v_j\|_{ -\frac12 - \alpha }=1$.

We conclude that $ u_\epsilon $ is bounded in $ H^{-\frac12 - \alpha }
( M ) $ for all $ \alpha > 0 $. But then similarly to the previous paragraph
$(u_\epsilon)_{\epsilon\to 0}$ is precompact in $H^{-\frac 12-\alpha}(M)$ for all $\alpha>0$.
Since every limit point has to be the (unique) solution to~\eqref{e:lapidus},
we see that $u_\epsilon$ converges as $\epsilon\to 0+$ in $H^{-\frac 12-\alpha}(M)$
to that solution.

As for continuity in $\omega$, we note that the above proof
gives the stronger statement
\begin{equation}
  \label{e:continuor}
(P-\omega_j-i\epsilon_j)^{-1}f\xrightarrow{ H^{-\frac12 -  } ( M)  } (P-\omega-i0)^{-1}f
\end{equation}
for all $\epsilon_j\to 0+$,
$\omega_j\to \omega$,
and
$|\omega_j|\leq \delta$.
\end{proof} 
In~\S\ref{s:lagreg} we will need the following
upgraded version of Lemma~\ref{l:lap}:
\begin{lemm}
\label{l:lapup}
Suppose that $ P $ satisfies \eqref{eq:assP}, \eqref{eq:dynaSC}, and~\eqref{e:no-spectrum}.
Let $ s < -\frac12 $ and
$  g \in H^{ s + 1 } ( M ) $, $ \WF ( g) \subset \Lambda^+ $, 
where $ \Lambda^+ $ is defined by \eqref{e:Lambda-pm-def}.
Then for $|\omega|\leq\delta$ the limit
\begin{equation}
\label{eq:lapup}  ( P - \omega - i \epsilon )^{-1} g 
\xrightarrow{ H^{s-} ( M)  } ( P - \omega - i 0 )^{-1} g,\quad
\epsilon\to 0+
\end{equation}
exists, and $ \WF ( ( P - \omega - i 0 )^{-1} g ) \subset 
\Lambda^+ $. In particular, for $k\geq 1$ and $ f \in C^\infty ( M ) $ 
the limit
\begin{equation}
  \label{e:lapup2}
( P - \omega - i \epsilon )^{-k}f 
\xrightarrow{ H^{-k + \frac12 - } ( M)  } ( P - \omega - i 0 )^{-k} f,\quad
\epsilon\to 0+ , 
\end{equation}
exists. Finally, $ ( P - \omega - i 0 )^{-1} f \in C^{k}_\omega ( [- \delta, \delta];H^{ -k- \frac12 - } ( M ) ) $
with $\partial_\omega^k( P - \omega - i 0 )^{-1} f=k!( P - \omega - i 0 )^{-k-1} f$.
\end{lemm}
\begin{proof}
We follow closely the proof of Lemma~\ref{l:lap} and put $ u_\epsilon := ( P - \omega- i \epsilon )^{-1} g $. 
Since $P-\omega-i\epsilon$ is elliptic
for every $\epsilon>0$, we have $u_\epsilon\in H^{s+1}(M)$ and $\WF(u_\epsilon)\subset \WF(g)\subset\Lambda^+$,
so it remains to establish uniformity as $\epsilon\to 0+$.
We use the following version of~\eqref{eq:uep2} (which follows
from the same proof): for every $A\in\Psi^0(M)$ with $\WF(A)\cap \Lambda^+=\emptyset$
there exists $\widetilde B\in\Psi^0(M)$ with $\WF(\widetilde B)\cap\Lambda^+=\emptyset$ such that
\begin{equation}
  \label{e:uep2-adv}
\|Au_\epsilon\|_{s'}\leq C\|\widetilde B g\|_{s'+1}+C\|u_\epsilon\|_{-N},\quad
s'>-\textstyle{1\over 2}
\end{equation}
where the constant $C$ does not depend on $\omega,\epsilon$. We also have the following
version of~\eqref{eq:uep1}: there exists $B'\in\Psi^0(M)$ with $\WF(B')\cap\Lambda^+=\emptyset$ such that
\begin{equation}
  \label{e:uep1-adv}
\|u_\epsilon\|_s\leq
C\|g\|_{s+1}+C\|B' g\|_{1}+C\|u_\epsilon\|_{-N},\quad
s<-\textstyle{1\over 2}.
\end{equation}
Here the norms $\|\widetilde Bg\|_{s'+1}$ and $\|B'g\|_1$
are finite since $\WF(g)\subset\Lambda^+$. From~\eqref{e:uep2-adv} and~\eqref{e:uep1-adv} we get
regularity for limit points of $u_{\epsilon_j}$ similarly to~\eqref{eq:uesj}:
\[
\exists \, \epsilon_j \to 0+ , \ u \in \mathscr D' ( M ) , \
 u_{\epsilon_j } \xrightarrow{ \mathscr D' ( M)  }  u  \quad
\Longrightarrow \quad u \in H^{ s } ( M ) , \ \ 
\WF ( u ) \subset \Lambda^+ .
\]
The existence of the limit~\eqref{eq:lapup} follows as in the proof of Lemma~\ref{l:lap},
replacing $-{1\over 2}$ by $s$ in Sobolev space orders; here
$u=(P-\omega-i0)^{-1}g$ is the unique solution to
$$
(P-\omega)u=g,\quad
\WF(u)\subset\Lambda^+.
$$
Iterating this argument, we get existence of the limit~\eqref{e:lapup2}
and continuous dependence of $(P-\omega-i0)^{-k}f\in H^{-k+{1\over 2}-}$
on $\omega\in [-\delta,\delta]$ similarly to~\eqref{e:continuor},
with $u=(P-\omega-i0)^{-k}f$ being the unique solution to
$$
(P-\omega)^k u=f,\quad
\WF(u)\subset\Lambda^+.
$$
It remains to show differentiability in~$\omega$.
For simplicity we assume that $\omega=0$ and show that
for $f\in C^\infty(M)$,
\begin{equation}
  \label{e:diffor}
\partial_\omega \big[(P-\omega-i0)^{-1}f\big]\big|_{\omega=0}
=(P-\omega-i0)^{-2}f\quad\text{in}\quad H^{-\frac 32-}.
\end{equation}
The case of higher derivatives is handled by iteration.
To show~\eqref{e:diffor} we denote $u_\epsilon(\omega):=(P-\omega-i\epsilon)^{-1}f$ and write for $\omega\neq 0$, with limits in $ H^{-\frac32 - } $
\begin{equation}
  \label{e:diffor2}
\begin{aligned}
{u_0(\omega)-u_0(0)\over\omega}
&=\lim_{\epsilon\to 0+}{u_\epsilon(\omega)-u_\epsilon(0)\over\omega}
=\lim_{\epsilon\to 0+}(P-\omega-i\epsilon)^{-1}(P-i\epsilon)^{-1}f
\\&=(P-\omega-i0)^{-1}(P-i0)^{-1}f.
\end{aligned}
\end{equation}
To show the last equality above we first note
that the family $(P-\omega-i\epsilon)^{-1}(P-i\epsilon)^{-1}f$ is precompact
in $H^{-{3\over 2}-\alpha}(M)$ for any $\alpha>0$ as follows from iterating~\eqref{e:uep1-adv}.
By~\eqref{e:uep2-adv} every limit point $u$ of this family as $\epsilon\to 0+$ satisfies
$P(P-\omega)u=f$, $\WF(u)\subset\Lambda$ and thus equals
$(P-\omega-i0)^{-1}(P-i0)^{-1}f$. Finally, letting $\omega\to 0$ in~\eqref{e:diffor2}
we get~\eqref{e:diffor}.
\end{proof}

\section{Lagrangian structure of the resolvent}
\label{lare}

In this section we describe the Lagrangian structure of the resolvent
refining the results of Haber--Vasy~\cite{hb} in our special case.
To start, we briefly review basic theory of Lagrangian distributions following~\cite[\S25.1]{H4}.

\subsection{Lagrangian distributions}
  \label{s:lagrangian-basic}

Let $M$ be a compact surface and $\Lambda_0\subset T^*M\setminus 0$
a conic Lagrangian submanifold without boundary. 
Denote by $I^s(M;\Lambda_0)\subset\mathcal D'(M)$ the space of Lagrangian distributions
of order~$s$ on $M$
associated to $\Lambda_0$. They have the following properties:
\begin{enumerate}
\item $I^s(M;\Lambda_0)\subset H^{-{1\over 2}-s-}(M)$;
\item for all $u\in I^s(M;\Lambda_0)$ we have $\WF(u)\subset\Lambda_0$;
\item if $\Lambda_1\subset \Lambda_0$ is an open conic subset
and $u\in I^s(M;\Lambda_0)$, then $u\in I^s(M;\Lambda_1)$
if and only if $\WF(u)\subset \Lambda_1$;
\item for all $A\in \Psi^k(M)$ and $u\in I^s(M;\Lambda_0)$
we have $Au\in I^{s+k}(M;\Lambda_0)$;
\item if additionally $\sigma(A)|_{\Lambda_0}=0$, then $Au\in I^{s+k-1}(M;\Lambda_0)$.
\end{enumerate}
Denote
$$
I^{s+}(M;\Lambda_0):=\bigcap_{s'>s} I^{s'}(M;\Lambda_0).
$$
A simple example on a torus (in the notation of \S \ref{exa}) is given by 
\begin{equation}
\label{eq:exala}
u ( x ) := ( x_1 - \tfrac{\pi}2 - i 0 )^{-1} \varphi(x), \ \ 
\varphi \in C^\infty_{\rm{c}} ( B ( 0 , 1 ) ) , \ \ 
u \in I^{0} ( \mathbb T^2 ; \Lambda_0^+ ) \subset H^{-\frac12 - } ( 
\mathbb T^2 ) , 
\end{equation}
where $ \Lambda_0^+ $ is given in \eqref{eq:P1}.

To define Lagrangian distributions we use Melrose's iterative 
characterization~\cite[Definition~25.1.1]{H4}:
$u\in\mathcal D'(M)$ lies in
$I^{s+}(M;\Lambda_0)$ if and only if $\WF(u)\subset\Lambda_0$ and
\begin{equation}
  \label{e:lagr-char}
A_1\dots A_\ell \, u\in H^{-{1\over 2}-s-}(M)\quad\text{for any}\quad
A_1,\dots,A_\ell\in\Psi^1(M),\
\sigma(A_j)|_{\Lambda_0}=0.
\end{equation}
Note that~\cite{H4} uses Besov spaces ${}^\infty H^s$, however this does not
make a difference in~\eqref{e:lagr-char} since $H^s\subset {}^\infty H^s\subset H^{s'}$
for all $s'<s$, see~\cite[Proposition~B.1.2]{H3}.

We also need oscillatory integral representations for Lagrangian distributions.
Assume that in some local coordinate system on $M$, $\Lambda_0$ is given by
\begin{equation}
  \label{e:lm-par}
\Lambda_0=\{(x,\xi)\colon x=\partial_\xi F(\xi),\ \xi\in\Gamma_0\}
\end{equation}
where $\Gamma_0\subset\mathbb R^2\setminus 0$ is an open cone and $F:\Gamma_0\to\mathbb R$ is
homogeneous of order~1. (Every Lagrangian can be locally written in this form after a change of base, $ x $, variables~-- see~\cite[Theorem 21.2.16]{H3}. Using a pseudodifferential partition of unity
we can write every Lagrangian distribution as a sum of expressions of the form~\eqref{e:lagros}.)
Then $u\in I^s(M;\Lambda_0)$ if and only if $u$ can be written (modulo a $C^\infty$ function)
as
\begin{equation}
  \label{e:lagros}
u(x)=\int_{\Gamma_0}e^{i(\langle x,\xi\rangle-F(\xi))}a(\xi)\,d\xi
\end{equation}
where $a(\xi)\in C^\infty(\mathbb R^2)$ is a symbol of order $s-{1\over 2}$, namely
\begin{equation}
\label{eq:symb}
|\partial^\alpha_\xi a(\xi)|\leq C_\alpha \langle\xi\rangle^{s-{1\over 2}-|\alpha|},\quad
\xi\in\mathbb R^2
\end{equation}
and $a$ is supported in a closed cone contained in $\Gamma_0$. See~\cite[Proposition~25.1.3]{H4}.
An equivalent way of stating~\eqref{e:lagros} is in terms of the Fourier transform $\hat u$:
$e^{iF(\xi)}\hat u(\xi)$ is a symbol, that is, satisfies estimates \eqref{eq:symb}.

We finally review properties of the principal symbol of a Lagrangian distribution,
used in the proof of Lemma~\ref{l:lagreg-plus} below, referring the reader to~\cite[Chapter~25]{H4}
for details.
The principal symbol of a Lagrangian distribution, $ u $, with values in half-densities,  $u\in I^s(M, \Lambda;  \Omega^{\frac12}_M )$, is the equivalence class 
$$ \sigma(u)\in S^{s+{1\over 2}}(\Lambda;\mathcal M_\Lambda \otimes \Omega_\Lambda^{1\over 2})/
S^{s-{1\over 2}}(\Lambda;\mathcal M_\Lambda \otimes \Omega_\Lambda^{1\over 2}),
$$
see \cite[Theorem~25.1.9]{H4}, where
\begin{itemize}
\item $\Omega_\Lambda^{1\over 2}$ is the line bundle of half-densities on $\Lambda$;
\item $\mathcal M_\Lambda$ is the Maslov line bundle; it has a finite number of prescribed local frames with ratios of any two prescribed frames given by a constant of absolute value one. Consequently it has a canonical inner product and does not enter into the calculations below;
\item $S^k(\Lambda;\mathcal M_\Lambda\otimes\Omega_\Lambda^{1\over 2})$ is the space
of sections in $C^\infty(\Lambda;\mathcal M_\Lambda\otimes\Omega_\Lambda^{1\over 2})$
which are symbols of order~$k$, defined using the dilation operator
$(x,\xi)\mapsto (x,\lambda\xi)$, $\lambda>0$, see the discussion on~\cite[page~13]{H4}.
In the parametrization~\eqref{e:lagros} we have
$\sigma(u |dx|^{\frac12} )=(2\pi)^{-\frac12}a(\xi)|d\xi|^{\frac12}$. The factor $|d\xi|^{\frac12}$
accounts for the difference in the order of the symbol.
\end{itemize}
If $P\in\Psi^\ell(M; \Omega_M^{\frac12} )$ satisfies $\sigma(P)|_{\Lambda}=0$ and $u\in I^s(M,\Lambda;\Omega_M^{\frac12} )$ then
\begin{equation}
  \label{e:transport-eqn}
  Pu\in I^{s+\ell-1}(M, \Lambda; \Omega_M^{\frac12} ),\quad
\sigma(Pu)= \tfrac 1 i L \sigma(u)
\end{equation}
where $L$ is a first order differential operator on $C^\infty(\Lambda;\mathcal M_\Lambda\otimes\Omega_\Lambda^{1\over 2})$ with principal part $H_p$. The equation \eqref{e:transport-eqn} is the {\em transport equation} for $P $ 
(the {\em eikonal equation} corresponds to $  \sigma ( P ) |_\Lambda = 0  $)~-- see~\cite[Theorem~25.2.4]{H4}. 
If $P$ is self-adjoint, then its subprincipal symbol is real-valued by~\cite[Theorem~18.1.34]{H3}
and thus by~\cite[(25.2.12)]{H4}
\begin{equation}
\label{eq:LLst} 
 L^* = -L  \quad \text{on } 
L^2 ( \Lambda; \mathcal M_\Lambda \otimes \Omega_\Lambda^{\frac12} ) . 
\end{equation}

\subsection{Lagrangian regularity}
  \label{s:lagreg}
  
We now establish Lagrangian regularity for elements in the range
of the operators $(P-\omega\mp i0)^{-1}$ constructed in~\S\ref{lap}:
\begin{lemm}
  \label{l:lagreg}
Suppose that $ P $ satisfies \eqref{eq:assP}, \eqref{eq:dynaSC}, and~\eqref{e:no-spectrum}.
Let $f\in C^\infty(M)$ and
$$
u^\pm(\omega):=(P-\omega\mp i0)^{-1}f\in H^{-{1\over 2}-}(M),\quad
|\omega|\leq\delta.
$$
Then $u^\pm(\omega)\in I^{0}(M ; \Lambda^\pm_\omega)$.
Moreover, the symbols of $u^\pm(\omega)$ depend smoothly on $\omega$:
\begin{equation}
  \label{e:lagreg}
 u^\pm(\omega)\in C^\infty_\omega\big([-\delta,\delta];I^{0}(M;\Lambda^\pm_\omega)\big)  , \end{equation}
where the precise meaning of~\eqref{e:lagreg} is explained in  { Lemma~\ref{l:lagreg-oi} below (\eqref{e:lagreg-oi-2} and Remark 2)}.
\end{lemm}
\Remark
Lemma~\ref{l:lagreg} is similar to the results
of Haber and Vasy~\cite[Theorem 1.7, Theorem 6.3]{hb}. There are two differences:
\cite{hb} makes the assumption that the Hamiltonian field $H_p$ is radial on $\Lambda^\pm_\omega$
(which is not true in our case) and it also does not prove smooth dependence of the symbols of $u^\pm(\omega)$ on $\omega$.
Because of these we give a self-contained proof of Lemma~\ref{l:lagreg} below,
noting that the argument is  simpler in our situation.


We focus on the case of $u^+(\omega)$, with
regularity of $u^-(\omega)$ proved by replacing $P, \, \omega$ with $-P, \, - \omega$, respectively. 
By Lemma~\ref{l:lapup} we have for every $k\geq 0$
\begin{equation}
  \label{e:lag-apriori}
u^+(\omega)\in C^k_\omega([-\delta,\delta];H^{-k-{1\over 2}-}(M)),\quad
\WF(\partial^k_\omega u^+(\omega))\subset \Lambda^+
\end{equation}
where the wavefront set statement is uniform in $\omega$.

To upgrade~\eqref{e:lag-apriori} to Lagrangian regularity, we use the criterion~\eqref{e:lagr-char},
applying first order operators $W$ and $D_\omega-Q$ to $u^+(\omega)$ (see Lemma~\ref{l:lagr-iter} below). Here,
\begin{equation}
\label{eq:defWQ}
W,Q\in\Psi^1(M),\quad
\sigma(W)=G_+,\quad
\sigma(Q)|_{\Lambda^+}=\Phi_+
\end{equation}
where $G_+$ is the defining function of $\Lambda^+$ constructed in Lemma~\ref{l:G-construction}
and $\Phi_+$ is defined in~\eqref{e:Phi-new-def}.
The operator $D_\omega-Q$, where $D_\omega:={1\over i}\partial_\omega$, is used to establish smoothness in~$\omega$.

Our proof uses the following corollary of~\eqref{eq:rad_sink}:
\begin{equation}
  \label{e:step}
\begin{gathered}
\text{if}\quad Z\in\Psi^{-1}(M),\quad
\sigma(Z)|_{\Lambda^+}=0,\quad
s<-\textstyle{1\over 2}
\quad\text{then}
\\
v\in\mathcal D'(M),\quad
\WF(v)\subset\Lambda^+,\quad
(P+Z-\omega)v\in H^{s+1}\quad\Longrightarrow\quad
v\in H^s.
\end{gathered}
\end{equation}
The addition of $Z$ does not change the validity of~\eqref{eq:rad_sink}
since it is a subprincipal term whose symbol vanishes on $\Lambda^+$,
see~\cite[Theorem~E.54]{res}.

We also use the following identity valid for any operators $A,B$ on $\mathcal D'(M)$: 
\begin{equation}
  \label{e:a-bit-of-algebra}
B^mA=\sum_{j=0}^m \binom{m}{j}(\ad^j_B A)B^{m-j}, \ \ \ \ 
\ad_B A := [ B, A ] ,\quad
 {\ad^0_BA:=A}. 
\end{equation}
The first step of the proof is to establish regularity with respect to powers of $W$:
\begin{lemm}
  \label{l:lagr-W}
Assume that $v\in\mathcal D'(M)$ satisfies for some $\ell\geq 0$
and $s<-{1\over 2}$
\begin{equation}
  \label{e:lagr-W}
\WF(v)\subset\Lambda^+,\quad
W^j(P-\omega)v\in H^{s+1}\quad\text{for}\quad j=0,\dots,\ell.
\end{equation}
Then $W^\ell v\in H^s$, where $  W $ is defined in \eqref{eq:defWQ}.
\end{lemm}
\begin{proof}
We argue by induction on~$\ell$. For $\ell=0$ the lemma follows immediately from~\eqref{e:step}.
We thus assume that $\ell>0$ and the lemma is true for all smaller values of $\ell$, in particular
$W^kv\in H^s$ for $0\leq k\leq \ell-1$.
Using~\eqref{e:a-bit-of-algebra} we write
\begin{equation}
  \label{e:lW-1}
W^\ell (P-\omega)=(P-\omega)W^\ell+\sum_{j=1}^\ell \binom{\ell}{j}(\ad^j_W P)W^{\ell-j}.
\end{equation}
We recall from Lemma~\ref{l:G-construction}
that near $\Lambda^+$ we have $H_{G_+}p=-a_+G_+$ where $a_+$ is homogeneous
of order~$-1$ and $a_+|_{\Lambda^+}=0$. Therefore for $j\geq 1$ we have
$H_{G_+}^j p=-(H_{G_+}^{j-1}a_+)G_+$ near $\Lambda^+$.
Motivated by this we take
$$
B_j\in \Psi^{-1}(M),\quad 
\sigma(B_j)=(-1)^{j-1}i^jH_{G_+}^{j-1}a_+, \quad 
1 \leq j \leq \ell .
$$
Then, for $1\leq j\leq \ell$
\begin{equation}
  \label{e:lW-2}
\ad^j_WP=B_jW+R_j,\quad
R_j\in \Psi^{-1}\quad\text{microlocally near }\Lambda^+.
\end{equation}
Combining~\eqref{e:lW-1} and~\eqref{e:lW-2} we get
\begin{equation}
  \label{e:lW-3}
(P-\omega)W^\ell=W^\ell (P-\omega)-\sum_{j=1}^\ell \binom{\ell}{j} (B_j W^{\ell+1-j}+R_j W^{\ell-j}).
\end{equation}
Applying both sides of~\eqref{e:lW-3} to $v$ and using that $W^kv\in H^s$ for $0\leq k\leq\ell-1$
and that $W^\ell (P-\omega)v\in H^{s+1}$
we get
$$
(P+\ell B_1-\omega)W^\ell v\in H^{s+1}.
$$
Since $\sigma(B_1)=ia_+$ vanishes on $\Lambda^+$, we apply~\eqref{e:step}
to conclude that $W^\ell v\in H^s$ as needed.
\end{proof}
Since $(P-\omega)u^+(\omega)=f\in C^\infty(M)$, Lemma~\ref{l:lagr-W} implies that
\begin{equation}
  \label{e:limo}
W^\ell u^+(\omega)\in H^{-{1\over 2}-}(M)\quad\text{for all}\quad\ell\geq 0.
\end{equation}
This can be generalized as follows:
\begin{equation}
  \label{e:upgrador}
A_1\dots A_\ell u^+(\omega)\in H^{-{1\over 2}-}(M)\quad\text{for all}\quad
A_1,\dots,A_\ell\in \Psi^1(M),\
\sigma(A_j)|_{\Lambda^+}=0.
\end{equation}
To see~\eqref{e:upgrador}, we argue by induction on~$\ell$.
We have $\sigma(A_j)=\tilde a_j G_+$ near $\WF(u^+(\omega))\subset\Lambda^+$
for some $\tilde a_j$ which is homogeneous of order~0.
Taking
$\widetilde A_j\in\Psi^0(M)$ with $\sigma(\widetilde A_j)=\tilde a_j$ we have
$$
A_j=\widetilde A_j W+\widetilde R_j\quad\text{where}\quad \widetilde R_j\in \Psi^0(M)\quad\text{microlocally near}\quad \WF(u^+(\omega)).
$$
Then we can write $A_1\dots A_\ell u^+(\omega)$
as the sum of two kinds of terms (plus a $C^\infty$ remainder):
\begin{itemize}
\item the term $\widetilde A_1\dots \widetilde A_\ell W^\ell u^+(\omega)$,
which lies in $H^{-{1\over 2}-}(M)$ by~\eqref{e:limo}, and
\item terms of the form $A'_1\dots A'_m u^+(\omega)$ where
$0\leq m\leq \ell-1$, $A'_j\in \Psi^1(M)$, and $\sigma(A'_j)|_{\Lambda^+}=0$,
which lie in $H^{-{1\over 2}-}(M)$ by the inductive hypothesis.
\end{itemize}
 {From~\eqref{e:upgrador} we can deduce (similarly to the proof of Lemma~\ref{l:lagreg-oi} below)
that $u^+(\omega)\in I^{0+}(M;\Lambda^+_\omega)$ for each $\omega\in[-\delta,\delta]$.
To obtain the smooth dependence of the symbol of $u^+(\omega)$ on~$\omega$
we generalize~\eqref{e:limo} by additionally applying powers of $D_\omega-Q$:}
\begin{lemm}
  \label{l:lagr-iter}
For all integers $\ell,m\geq 0$ we have
\begin{equation}
  \label{e:lagr-iter}
W^\ell (D_\omega-Q)^m u^+(\omega)\in H^{-{1\over 2}-}(M),\quad
|\omega|\leq \delta,
\end{equation}
and the corresponding norms are bounded uniformly in $\omega$.
\end{lemm}
\begin{proof}
We argue by induction on $m$, with the case $m=0$ following from~\eqref{e:limo}. Put
$$
u_j(\omega):=(D_\omega-Q)^j u^+(\omega)\in\mathcal D'(M),\quad
0\leq j\leq m.
$$
By~\eqref{e:lag-apriori} we have $\WF(u_j(\omega))\subset\Lambda^+$ for all~$j$.
Moreover, by the inductive hypothesis
\begin{equation}
  \label{e:liter-ind}
W^\ell u_j(\omega)\in H^{-{1\over 2}-}(M)\quad\text{for all}\quad \ell,\
0\leq j\leq m-1.
\end{equation}
Put
$$
Y:=[P-\omega,D_\omega-Q]=-i-[P,Q]\in\Psi^0(M)
$$
and note that since $ \sigma ( Q )|_{\Lambda^+} = \Phi_+ $ and $H_p\Phi_+\equiv 1$ on $\Lambda^+$ by~\eqref{e:Phi-prop-1},
\begin{equation}
  \label{e:Y-vanisher}
\sigma(Y)|_{\Lambda^+}=0.
\end{equation}
Moreover, by~\eqref{e:Phi-prop-1} we have $H_{G_+}\Phi_+\equiv 0$ on $\Lambda^+$,
thus the Hamiltonian vector field $H_{\Phi_+}$ is tangent to $\Lambda^+$. This implies that
\begin{equation}
  \label{e:Y-vanisher-2}
\sigma(\ad_Q^j Y)=(-i)^j H_{\Phi_+}^j \sigma(Y)\equiv 0\quad\text{on}\quad \Lambda^+\quad\text{for all}\quad
j\geq 0.
\end{equation}
Applying~\eqref{e:a-bit-of-algebra} with $A:=P-\omega$ and $B:=D_\omega-Q$ 
to $u^+(\omega)$ we get
\begin{equation}
\label{e:rightor}
(P-\omega)u_m(\omega)=(D_\omega-Q)^m f+
\sum_{j=1}^m (-1)^{j-1}\binom{m}{j}(\ad_Q^{j-1}Y)u_{m-j}(\omega).
\end{equation}
Since $f\in C^\infty$ does not depend on~$\omega$, we have $(D_\omega-Q)^mf\in C^\infty$.
Next, by the inductive hypothesis~\eqref{e:liter-ind} we have
$W^\ell u_{m-j}(\omega)\in H^{-{1\over 2}-}$ for all $\ell\geq 0$ and $1\leq j\leq m$.
Arguing similarly to~\eqref{e:upgrador} and using~\eqref{e:Y-vanisher-2} we see
that $W^\ell (\ad_Q^{j-1}Y)u_{m-j}(\omega)\in H^{{1\over 2}-}$ as well
(here $\ad_Q^{j-1}Y\in \Psi^0(M)$ which explains the stronger regularity).
Thus \eqref{e:rightor} implies
$$
W^\ell (P-\omega)u_m(\omega)\in H^{{1\over 2}-}(M)\quad\text{for all}\quad\ell\geq 0.
$$
Now Lemma~\ref{l:lagr-W} gives $W^\ell u_m(\omega)\in H^{-{1\over 2}-}$ for all $\ell\geq 0$
as needed.

Finally, uniformity of~\eqref{e:lagr-iter} in $\omega$ follows immediately
from the proof since the estimates~\eqref{e:lag-apriori} and~\eqref{eq:rad_sink}
that we used are uniform in~$\omega$.
\end{proof}
We now deduce from Lemma~\ref{l:lagr-iter} that $u^+(\omega)$
has microlocal oscillatory integral representations~\eqref{e:lagros}
with symbols depending smoothly on~$\omega$. This shows the weaker version of~\eqref{e:lagreg}
with $I^0$ replaced by $I^{0+}$.
\begin{lemm}
  \label{l:lagreg-oi}
Assume that $\mathcal U\subset T^*M\setminus 0$ is an open conic set such that
$\Lambda^+_\omega\cap \mathcal U$ are given in the form~\eqref{e:phase-der-1}
in some local coordinate system on~$M$:
\begin{equation}
  \label{e:lagreg-oi-1}
\Lambda^+_\omega\cap\mathcal U=\{(x,\xi)\colon x=\partial_\xi F(\omega,\xi),\ \xi\in\Gamma_0\},\quad
|\omega|\leq\delta
\end{equation}
where $\xi\mapsto F(\omega,\xi)$ is homogeneous of order~1 and $\Gamma_0\subset\mathbb R^2\setminus 0$
is an open cone. Let $A\in\Psi^0(M)$, $\WF(A)\subset\mathcal U$.
Then,
\begin{equation}
  \label{e:lagreg-oi-2}
Au^+(\omega,x)=\int_{\Gamma_0}e^{i(\langle x,\xi\rangle-F(\omega,\xi))} a(\omega,\xi)\,d\xi
+C^\infty_{\omega,x},\quad
|\omega|\leq\delta
\end{equation}
where $a(\omega,\xi)$ is a smooth in $\omega$ family of symbols of order $-{1\over 2}+$ in $\xi$ supported in a closed cone inside $\Gamma_0$, see~\eqref{eq:symb}.
\end{lemm}
\Remarks
1. The statement \eqref{e:lagreg-oi-2} means that 
$ u^+ ( \omega ) $ can be represented as~\eqref{e:lagros}, 
{\em microlocally} in every closed cone contained in $ \mathcal U $.

\noindent
2. When \eqref{e:lagreg-oi-2} holds for every choice of parametrization~\eqref{e:lagreg-oi-1} we write 
\[ u^+(\omega)\in C^\infty_\omega\big([-\delta,\delta];I^{0+}(M;\Lambda^+_\omega)\big)  , \]
with the analogous notation in the case of $ u^- ( \omega ) $. That explains the statement of Lemma~\ref{l:lagreg}. 
\begin{proof}
Since $(P-\omega)u^+(\omega)=f\in C^\infty(M)$,
it follows from Lemma~\ref{l:lagr-iter} that for all $m,\ell,r\geq 0$
$$
(D_\omega-Q)^mW^\ell (P-\omega)^r u^+(\omega)\in H^{-{1\over 2}-}(M)
$$
This can be generalized as follows:
\begin{equation}
  \label{e:oi-1}
(D_\omega-Q(\omega))^mA_1(\omega)\dots A_\ell(\omega)u^+(\omega)\in H^{-{1\over 2}-}(M)
\end{equation}
for all $m$ and all $A_1(\omega),\dots,A_\ell(\omega),Q(\omega)\in\Psi^1(M)$ depending smoothly
on~$\omega\in [-\delta,\delta]$ and such that $\sigma(A_j(\omega))|_{\Lambda^+_\omega}=0$,
$\sigma(Q(\omega))|_{\Lambda^+_\omega}=\Phi_+$.
The proof is similar to the proof of~\eqref{e:upgrador}, using the decomposition
$$
\begin{gathered}
A_j(\omega)=A'_j(\omega)W+A''_j(\omega)(P-\omega)+R_j(\omega)\\
\text{where}\quad
R_j(\omega)\in\Psi^0\quad\text{microlocally near}\quad \WF(u^+(\omega))
\end{gathered}
$$
for some $A'_j(\omega),A''_j(\omega)\in \Psi^0(M)$ depending smoothly
on $\omega\in [-\delta,\delta]$.

Since $\WF(A\partial^k_\omega u^+(\omega))\subset \Lambda^+\cap p^{-1}([-\delta,\delta])\cap \mathcal U$
for all $k$, by
the Fourier inversion formula we can write $Au^+(\omega)$ in the form~\eqref{e:lagreg-oi-2}
for some $a(\omega,\xi)$ which is smooth in $\omega,\xi$ and supported in $\xi\in\Gamma_1$
where $\Gamma_1\subset\Gamma_0$ is some closed cone.
It remains to show the following growth bounds as $\xi\to \infty$: for every $\varepsilon>0$
\begin{equation}
  \label{e:derb-l2}
\langle\xi\rangle^{-{1\over 2}+|\alpha|-\varepsilon} \partial^m_\omega \partial^\alpha_\xi a(\omega,\xi)
\in L^\infty_\omega([-\delta,\delta]; L^2_\xi(\mathbb R^2)).
\end{equation}
(From~\eqref{e:derb-l2} one can get $L^\infty_\xi$ bounds using Sobolev embedding
as in the proof of~\cite[Proposition~25.1.3]{H4}.)

Denote by $\mathcal I(a)$ the integral on the right-hand side of~\eqref{e:lagreg-oi-2}.
By Lemma~\ref{l:phase-der} we have
$\partial_\omega F(\omega,\xi)=-\Phi_+(\partial_\xi F(\omega,\xi),\xi)$, therefore
we may take $Q(\omega):=-\partial_\omega F(\omega, D_x)$ to be a Fourier multiplier.
The operators
$$
A_{jk}(\omega):=D_{x_k}\big((\partial_{\xi_j}F)(\omega,D_x)-x_j\big),\quad
j,k\in \{1,2\},
$$
lie in $\Psi^1$ and
satisfy $\sigma(A_{jk}(\omega))|_{\Lambda^+_\omega}=0$. We have
$$
(D_\omega-Q(\omega))\mathcal I(a)=\mathcal I(D_\omega a),\quad
A_{jk}(\omega)\mathcal I(a)=\mathcal I(\xi_k D_{\xi_j}a).
$$
Also, if $\mathcal I(a)\in H^{-{1\over 2}-}$ uniformly in $\omega$, then
$\langle\xi\rangle^{-{1\over 2}-\varepsilon}a(\omega,\xi)\in L^\infty_\omega([-\delta,\delta];L^2_\xi(\mathbb R^2))$.
Applying~\eqref{e:oi-1} with the operators $D_\omega-Q(\omega)$ and $A_{jk}(\omega)$
we get~\eqref{e:derb-l2}, finishing the proof.
\end{proof}
We finally show the stronger statement of Lemma~\ref{l:lagreg} (with $I^0$ instead of $I^{0+}$)
using the transport equation satisfied by the principal symbol:
\begin{lemm}
  \label{l:lagreg-plus}
We have
$$
u^+(\omega)\in 
C^\infty_\omega\big([-\delta,\delta];
I^0(M;\Lambda_\omega^+)
\big),
$$
that is~\eqref{e:lagreg-oi-2} holds where $a(\omega,\xi)$ is a symbol of order $-{1\over 2}$ in~$\xi$.
\end{lemm}
\begin{proof}
In our setting $ P \in \Psi^0 ( M ) $ is self-adjoint with respect to a  smooth density on~$ M$ -- see \eqref{eq:assP}. 
Using that density to trivialize the half-density bundle we 
obtain a self-adjoint operator
$P\in \Psi^0(M;\Omega^{1\over 2}_M)$.

Let $a^+\in S^{{1\over 2}+}(\Lambda^+_\omega;\mathcal M_{\Lambda^+_\omega} \otimes \Omega_{\Lambda^+_\omega}^{1\over 2})$ be a representative of $\sigma(u^+(\omega))$.
Using the transport equation~\eqref{e:transport-eqn}
and
$(P-\omega)u^+(\omega)=f\in C^\infty(M)$, we have
\begin{equation}
  \label{e:tbone}
b^+:=La^+\in S^{-{3\over 2}+}(\Lambda^+_\omega;\mathcal M_{\Lambda^+_\omega} \otimes \Omega_{\Lambda^+_\omega}^{1\over 2}),
\end{equation}
where $L$ is a first-order differential operator on
$C^\infty(\Lambda^+_\omega;\mathcal M_{\Lambda^+_\omega} \otimes \Omega_{\Lambda^+_\omega}^{1\over 2})$
with principal part given by $H_p$ and $L^*=-L$
by~\eqref{eq:LLst}. 

We trivialize $\Omega^{1\over 2}_{\Lambda^+_\omega}$ using the density $\nu^+_\omega$ constructed in Lemma~\ref{l:density}
and write
$$
a^+=\tilde a^+\sqrt{\nu^+_\omega},\quad
b^+=\tilde b^+\sqrt{\nu^+_\omega}.
$$
where $\tilde a^+\in S^{0+}(\Lambda^+_\omega;\mathcal M_{\Lambda^+_\omega})$,
$\tilde b^+\in S^{-2+}(\Lambda^+_\omega;\mathcal M_{\Lambda^+_\omega})$. By~\eqref{e:tbone} we have
\begin{equation}
  \label{e:hradish}
(H_p+V)\tilde a^+=\tilde b^+
\end{equation}
where $H_p$ naturally acts on sections of the locally constant bundle
$\mathcal M_{\Lambda^+_\omega}$ and $V\in C^\infty(\Lambda^+_\omega)$
is homogeneous of order~$-1$. Moreover, since $L^*=-L$ 
we have
$$
\Re V=\tfrac{1}{2}(\mathcal L_{H_p}\nu^+_\omega)/\nu^+_\omega=0
$$
using Lemma~\ref{l:density}. 

By~\eqref{e:hradish} for all $(x,\xi)\in\Lambda^+_\omega$ and $t\geq 0$ we have
\begin{equation}
  \label{e:flourish}
\tilde a^+(x,\xi)= {\big(}e^{-t(H_p+V)}\tilde a^+ {\big)}(x,\xi)+\int_0^t  {\big(}e^{-s(H_p+V)}\tilde b^+ {\big)}(x,\xi)\,ds.
\end{equation}
Since $\Re V=0$ we have $|e^{-t(H_p+V)}\tilde a^+(x,\xi)|=|\tilde a^+(e^{-tH_p}(x,\xi))|$
and same is true for $\tilde b^+$.

Take $(x,\xi)\in \Lambda^+_\omega$ with $|\xi|$ large.
As in~\eqref{e:lynmar} choose $t\geq 0$, $t\sim |\xi|$, such that $e^{-tH_p}(x,\xi)\in S^*M$;
we next apply~\eqref{e:flourish}. The first term on the right-hand side is bounded uniformly
as $\xi\to\infty$. Same is true for the second term since the function under the integral
is $\mathcal O((t-s)^{-2+})$.
It follows that $\tilde a^+(x,\xi)$ is bounded as $\xi\to\infty$.

Since $[\xi\partial_\xi,H_p+V]=-H_p-V$, we have for all $j$
\begin{equation}
  \label{e:heald}
(H_p+V)(\xi\partial_\xi)^j \tilde a^+=(\xi\partial_\xi+1)^j \tilde b^+\in S^{-2+}(\Lambda^+_\omega;\mathcal M_{\Lambda^+_\omega}).
\end{equation}
It follows that $(H_p+V)^\ell (\xi\partial_\xi)^j\tilde a^+=\mathcal O(\langle\xi\rangle^{-\ell})$
for all $j,\ell$: the case $\ell=0$ follows from~\eqref{e:flourish} applied to~\eqref{e:heald}
and the case $\ell\geq 1$ follows directly from~\eqref{e:heald}.
Since $\xi\partial_\xi$ and $H_p$ form a frame on $\Lambda^+_\omega$,
we have $\tilde a^+\in S^{0}(\Lambda^+_\omega;\mathcal M_{\Lambda^+_\omega})$
which implies that $u^+_\omega\in I^0(M;\Lambda^+_\omega)$.
\end{proof}

\Remark It is instructive to consider the transport equation 
\eqref{e:hradish} in the microlocal model used in \cite{SC}: near 
a model sink $ \Lambda^+_\omega = \{ ( -\omega , x_2 ; \xi_1 , 0 ) : \xi_1 > 0 \}
\subset T^* ( \RR_{x_1} \times \mathbb S^1_{x_2} ) \subset 0 $ (see the global examples in 
\S \ref{exa}) we consider $ p ( x, \xi ) := \xi_1^{-1} \xi_2 - x_1 $.
We are then solving $( p ( x , D )-\omega) u^+ ( \omega ) \equiv 0 $ 
microlocally near $ \Lambda^+_\omega $ (see \cite[Definition E.29]{res}) and  
for that we expand the symbol on $ u^+_\omega $ into Fourier modes in $ x_2 $,
\[
u^+_\omega ( x ) = \frac{1}{2\pi}  \int_\RR 
\sum_{ n \in \ZZ }  \hat a_\omega^+ ( n , \xi_1 ) e^{ i (x_1+\omega) \xi_1 }e^{inx_2} \,d \xi_1, \ \  a_\omega^+ = \sum_{n\in \ZZ} \hat  a_\omega^+ ( n 
, \xi_1 )e^{inx_2} | d \xi_1 dx_2 |^{\frac12} .
\] 
The Fourier coefficients should satisfy
$ ( \xi_1^{-1} n + D_{\xi_1}  ) \tilde a_\omega^+ ( n , \xi_1 ) = 0 $ for $ \xi_1 > 1 $ and $ \tilde a_+^\omega ( n, \xi_1 ) = 0 $ for $ \xi_1 < -1 $.
Hence the symbol is given by 
\[   
a_\omega^+=\tilde a^+ (\omega) |d x_2 d\xi_1|^{1\over 2},\quad
\tilde a^+ ( x_2, \xi_1 ) = \sum_{ n \in \ZZ } \xi_1^{ -i n  } 
a_n(\omega) e^{i n x_2} , \quad
    a_n (\omega) = \mathcal O ( \langle n\rangle^{-\infty} ). \]
Hence, the symbol is very ``non-classical" in the sense that it does not have an expansion in powers of $ \xi_1 $.  In the general case  {an analogous conclusion} follows from the structure of~\eqref{e:hradish}.

\section{An asymptotic result}
\label{asr}

We now place ourselves in the setting of Lemma \ref{l:lagreg} and 
assume that $ u ( \omega ) \in C^\infty_\omega ( [ - \delta, \delta ] ;
I^{ 0} ( M ; \Lambda_\omega )) $ in the sense described in Lemma~\ref{l:lagreg-plus},
where $\Lambda_\omega=\Lambda^+_\omega$ or $\Lambda_\omega=\Lambda^-_\omega$.
We are interested in the asymptotic behaviour as $t\to\infty$ of
\begin{equation}
\label{eq:Lagrancon1}
I ( t ) := 
\int_0^t  \int_\RR e^{ - i s \omega }   \varphi ( \omega )u ( \omega ) \,  d \omega ds \in\mathcal D'(M), \ \ \varphi \in \CIc ( (-\delta,\delta) ).
\end{equation}
We have the following local asymptotic result. 
\begin{lemm}
\label{lem}
Suppose that $ u ( \omega ) \in \mathcal D' ( \RR^2 ) $ is given by 
\begin{equation}
\label{eq:defxy} 
\begin{gathered} u ( \omega ) =  u ( \omega, x )  = \frac{1}{ (2 \pi)^2 } \int_{\Gamma_0} 
e^{ i ( \langle x , \xi \rangle - F ( \omega , \xi ) ) }  
a (\omega  ,  \xi ) \,d \xi , 
\end{gathered}
\end{equation}
where $\Gamma_0$,  $F $, and $ a $ satisfy the general conditions in \eqref{e:lagreg-oi-2}. 
Suppose also that
\begin{equation}
\label{eq:assFo}
\varepsilon \partial_\omega F ( \omega, \xi ) < 0 , \quad
\varepsilon = \pm,   \quad
\xi \in \Gamma_0,\quad
|\omega|\leq\delta .
\end{equation}
Then as $t\to\infty$,
\begin{equation}
\label{eq:Lagrancon2}
\begin{gathered} 
I ( t ) = u_\infty   + b ( t ) + v ( t ) ,  
 \ \ \|  {b} ( t ) \|_{ H^{ \frac 12 - } } \leq C, \ \ v ( t ) \to 0 \text{ in $ H^{  - \frac 12 - } ( \RR^2 )$}, \\
u_\infty = \left\{ \begin{array}{ll}     2 \pi \varphi ( 0 ) u( 0 ) ,
&  \varepsilon = +;  \\
\ \ \ \ \ 0, & \varepsilon = - .
\end{array} \right.
\end{gathered}
\end{equation}
\end{lemm}
\begin{proof}
We start by remarking
that
we can assume that the amplitude $ a $ is 
supported away from $ \xi = 0 $. 
The remaining contribution can be absorbed into $ b ( t ) $: if $ a = a (\omega, \xi ) = 0 $ for $ |\xi| > C $ then
\begin{equation*}
\begin{split} 
\widehat w ( t, \xi ) & :=\int_0^t \int_\RR  e^{ - i s \omega } e^{ - i F ( \omega, \xi )  } 
a ( \omega, \xi ) \varphi ( \omega ) d \omega ds \\
& = 
\int_0^t \int_\RR  \left[( 1 + s^2)^{-1} ( 1 + D_\omega^2 ) e^{ - i s \omega }\right] 
e^{  -i  F ( \omega, \xi  )   } 
a ( \omega, \xi ) \varphi ( \omega ) d \omega ds ,
\end{split} \end{equation*}
which by integration by parts in $ \omega $ is bounded in $ t $ and
compactly supported in $ \xi $. 

 {Since $u(\omega,x)$ has nice structure on the Fourier transform side it is natural to} consider the Fourier transform of $ x \mapsto I ( t  ) ( x )  $, $J ( t ,\xi ) := \mathcal F_{ x \to \xi } { I ( t ) }$, 
 where 
\begin{equation}
\label{eq:defJt} 
J ( t,\xi  ) = \frac 1 h \int_0^{ h t }  \int_\RR 
 e^{ - \frac i h  ( F ( \omega , \eta )  + r \omega  ) }  a ( \omega,  \eta /h )  \varphi ( \omega ) \, d \omega dr ,  \quad \xi = \frac \eta h , \ \ \eta \in \mathbb S^{1} . 
\end{equation}
 {From the assumptions on $a$ we have $J(t,\xi)=0$ unless $\eta\in\Gamma_1$, where
$\Gamma_1\subset \Gamma_0$ is a closed cone.}
The phase in $ J ( t ) $ is stationary when
\begin{equation}
\label{eq:cretin}  
 \omega = 0, \ \  r = r ( \eta ) := -  \partial_\omega F ( 0, \eta ) .
 \end{equation}
From \eqref{eq:assFo}, $ \partial_\omega F ( \omega , \eta ) \neq 0 $ and this means that for some $ \gamma > 0 $, 
\begin{equation}
\label{eq:lowFom}   
| r + \partial_\omega F( \omega, \eta ) | > c \langle r \rangle  , \ \ 
\eta \in \mathbb S^{1}\cap \Gamma_1 , \ \ | \omega | \leq \delta,   
\ \ |r| \notin ( \gamma, 1/\gamma ) .
\end{equation}
Let $ \chi \in \CIc ( ( \gamma/2 , 2/\gamma) ; [ 0, 1 ] ) $ 
be equal to $ 1 $ on $ ( \gamma , 1/\gamma ) $.
Using integration by parts based on 
\[  {h^N}\left( - ( r+ \partial_\omega F ( \omega , \eta ) )^{-1} D_\omega \right)^N 
e^{ - \frac i h ( F ( \omega , \eta )+r\omega) } =  e^{ - \frac i h ( F( \omega , \eta ) +r\omega) } , \]
and \eqref{eq:lowFom} we see that, by taking $ N \geq 2 $, 
\[
\begin{split} &  \frac1h \int_0^{h t }   \int_\RR ( 1 - \chi( r ) )
 e^{  - \frac i h ( F ( \omega , \eta )  + r \omega ) }  a ( \omega , \eta /h )  \varphi ( \omega ) \, d \omega dr 
= \mathcal O ( h^{N-  1 } ) , 
\end{split}
\]
uniformly in $t\geq 0$. Hence, for all $N$
\[
\begin{gathered} J ( t  ) = \widetilde J ( t  ) +\mathcal F_{ x \mapsto \xi }  u_0 ( t ), \ \ 
\sup_{t\geq 0}\| u_0 ( t ) \|_{H^N} \leq C_N, \\
\widetilde J  ( t,\xi  ) := \frac 1 h  \int_0^{ h t }   \int_\RR \chi ( r ) 
 e^{  - \frac i h ( F ( \omega , \eta )  + r \omega ) }  a ( \omega , \eta /h   )  \varphi ( \omega )  \,d \omega dr , \quad  
 \xi = \frac \eta h , \ \eta \in \mathbb S^{1} . \end{gathered}
\]
When $ ht   \geq 2/\gamma $, we have $ \widetilde J ( t,\xi  ) = 
\widetilde J  ( \infty,\xi  ) $ due to the support property of $\chi$.
In particular this implies that $\widetilde J(t,\xi)\to \widetilde J(\infty,\xi)$ as $t\to\infty$ pointwise in $\xi$.
We apply the standard 
method of stationary phase to $\widetilde J(\infty)$ noting that
\[  
 - \partial^2_{ \omega, r } ( F ( \omega, \eta ) + r \omega ) = 
\begin{bmatrix} -  \partial_\omega^2 F & - 1 \\
-1 & 0 \end{bmatrix} , \ \ \ \sgn \partial^2_{ \omega, r } ( F ( \omega, \eta ) - r \omega )  = 0.
 \]
Therefore
\begin{equation}
  \label{e:asyy}
\widetilde J   ( \infty ,\xi  ) = \left\{ \begin{array}{ll}
2 \pi a ( 0 , \xi) \varphi ( 0 ) e^{- i  F ( 0, \xi ) } + \mathcal O ( \langle \xi \rangle^{- \frac 32 + } ) ,
& \partial_\omega F ( 0 , \xi) <  0 , \\
\ \ \ \ \ \ \ \ \ \ \ \ \mathcal O ( \langle \xi\rangle^{-\infty } ), & \partial_\omega F ( 0 , \xi )  >  0.
\end{array} \right. 
\end{equation}
Hence to obtain \eqref{eq:Lagrancon2} all we need to show is that $ \widetilde J   ( t,\xi  ) = \mathcal O ( 
\langle \xi\rangle ^{- \frac12 + } ) $ uniformly in $ t $ as then by dominated convergence,
\[
  \begin{split} \langle \xi\rangle ^{ - \frac 12  -  }  \widetilde J   (t  ) 
&  \xrightarrow{ L^2 ( \RR^2,  d\xi ) } \langle \xi \rangle^{   -\frac 12  - } \widetilde J   ( \infty   ) , \ \ \ t \to +\infty ,  \end{split} 
\]
that is, 
\[
\widetilde I(t):=\mathcal F^{-1} _{\xi \to x}  \widetilde J ( t) 
\xrightarrow{H^{ -\frac12- } ( \RR^2 ) }
\mathcal F^{-1} _{\xi \to x} \widetilde J_\infty ( t ) , \ \ \ t \to + \infty .
\]
 {Here the $\mathcal O(\langle \xi\rangle^{-\frac 32+})$ remainder in~\eqref{e:asyy}
can be put into $b(t)$ in~\eqref{eq:Lagrancon2}.}

The uniform boundedness of $\widetilde J(t,\xi)$ follows from 
the following simple lemma:
\begin{lemm}
\label{l:trivial}
Suppose that $ A= A ( s , \omega ) \in \CIc ( \RR^2 ) $ and 
$ G \in \CI ( \RR; \RR ) $. Then as $h\to 0$
\begin{equation}
\label{eq:trivial}
L(h):= \int_0^\infty \int_\RR e^{ \frac i h ( G ( \omega ) + s \omega ) } 
A ( s, \omega ) \, d\omega ds = \mathcal O ( h \log (1/h)) . 
\end{equation}
\end{lemm}
\begin{proof}
We define
\[
B ( \sigma, \omega ) := \int_0^\infty e^{  i s \sigma } A ( s , \omega ) \,
ds , \ \  B ( \sigma , \omega ) =  i \sigma^{-1}{ A ( 0 , \omega ) } + 
\mathcal O ( \sigma^{-2 }) , \ \ |\sigma| \to \infty. 
\]
Hence,
\[
\begin{split}  L ( h ) & = \int_\RR e^{ \frac i h G ( \omega ) } B \left( \frac{ \omega} h , 
\omega \right) d \omega  = h \int_\RR e^{ \frac i h G ( h w ) } 
B ( w , h w ) \,d w \\
& = \mathcal O ( h ) \int_{|w| \leq C/h} \,\frac{ dw}{  1 + |w|} = 
\mathcal O ( h \log (1/h ) ), \end{split}
\]
proving \eqref{eq:trivial}. (In fact we see that the estimate is sharp: if we take $ G \equiv 0 $ and $ A $ which is {\em odd} in $ \omega $
one does have logarithmic growth.)
\end{proof}
To use the lemma to show the bound $ \widetilde J   ( t,\xi  ) = \mathcal O ( 
\langle \xi\rangle ^{- \frac12 + } ) $, uniformly in $ t\geq 0 $,
it suffices to consider the case $ht\leq 2/\gamma$,
since otherwise $\widetilde J(t,\xi)=\widetilde J(\infty,\xi)$.
As before, we write $\xi=\eta/h$ where $\eta\in\mathbb S^1$. Then
\[
\widetilde J (t,\xi) = \frac1h \int_0^\infty \int_\RR 
e^{ \frac i h ( s\omega- ht\omega - F ( \omega, \eta )  ) }
\chi ( ht - s ) a ( \omega , \eta /h ) \varphi ( \omega )\, d\omega d s.
\] 
We now apply Lemma~\ref{l:trivial} with $ A  ( s, \omega ) := h^{\alpha - \frac12} 
 \chi ( ht - s ) a ( \omega , \eta /h ) \varphi ( \omega )$, $ \alpha >0 $ (and
 arbitrary) 
   and 
$ G ( \omega ) = -ht\omega-F ( \omega , \eta )$ to obtain, 
$   \widetilde J  ( t) = \mathcal O ( h^{\frac12 - \alpha } \log(1/h)  ) = 
\mathcal O ( \langle \xi \rangle^{-\frac12 + 2 \alpha } ) $ which concludes the proof.
\end{proof}

\section{Proof of the Main Theorem}

In the approach of \cite{SC} the decomposition of $ u ( t ) $ is obtained using 
\eqref{eq:uoft} and proving that for $ \varphi $ supported in a neighbourhood of 
$ 0 $, 
\begin{equation}
  \label{e:scarab}
  P^{-1} ( e^{ - i tP  } - 1 ) \varphi ( P ) f
\xrightarrow{ H^{-\frac12 -  } ( M)  }  -( P - i 0 ) ^{-1}\varphi( P )  f , \quad
t \longrightarrow \infty ,
\end{equation}
which makes formal sense if we think in terms of distributions.  
The rigorous argument requires finer aspects of Mourre theory developed by 
Jensen--Mourre--Perry~\cite{jemp}. 

Here we take a more geometric approach and use Lemma~\ref{l:lap} and~\ref{l:lagreg}
to study the behaviour of $ u ( t ) $. Fix $\delta>0$ small enough so that the results of~\S\ref{s:sink-source},
as well as~\eqref{e:no-spectrum}, hold.
Fix $\varphi\in C^\infty_{\rm{c}}((-\delta,\delta))$ such that $\varphi=1$ near 0.
By~\eqref{eq:uoft}, the spectral theorem, 
and Stone's formula (see for instance \cite[Theorem B.8]{res}) we have 
\begin{equation}
\label{eq:uoft1}
\begin{aligned}
  u( t) &=
  -i\int_0^t e^{-isP}\varphi(P)f\,ds
  +P^{-1}(e^{-itP}-1)(1-\varphi(P))f
   \\&=
  \frac{1}{ 2 \pi  }\int_0^t  \int_\RR  e^{ - i s \omega } \varphi ( \omega )
(u^-(\omega)-u^+(\omega))\, d \omega ds
+ b_1 ( t) ,
\end{aligned}
 \end{equation}
where  $ \| b_1 ( t )\|_{L^2 } \leq C $ for all $t\geq 0$ and
$u^\pm(\omega):=(P-\omega\mp i0)^{-1}f\in H^{-1/2-}(M)$ are defined in
Lemma~\ref{l:lap}.

By Lemma~\ref{l:lagreg} we have
$u^\pm ( \omega )  \in C^\infty_{\omega} ( [ - \delta ,  \delta ] ;  I^{0} ( M; \Lambda^\pm_\omega ))$.
The main result \eqref{eq:SC2}, \eqref{eq:DZ1} then follows from Lemma~\ref{lem}.
Here we use a pseudodifferential partition
of unity to write $u^\pm(\omega)$ as a finite
sum of oscillatory integrals~\eqref{eq:defxy} and the geometric condition~\eqref{eq:assFo} follows from
Lemmas~\ref{l:phase-der} and~\ref{l:Phi-sign}.
We obtain $u_\infty=-u^+(0)$ which is consistent with~\eqref{e:scarab}.

\medskip\noindent\textbf{Acknowledgements.}
This note is a result of a ``groupe de travail'' on \cite{SC} conducted in Berkeley in February and March of 2018. We would like to thank the participants of that seminar and in particular Thibault de Poyferr\'e for explaining the fluid mechanical motivation to us.
Thanks go also to Andr\'as Vasy for a helpful discussion of results of
\cite{hb}. We are also grateful to Micha\l{} Wrochna for pointing out to us
a mistake in Lemma~\ref{l:sink-established}~-- see 
the remark following
that lemma~--  {and to the anonymous referee for many suggestions to improve the manuscript.}
This research was conducted during the period SD served as
a Clay Research Fellow and MZ was supported
by the National Science Foundation grant DMS-1500852 and by a Simons Fellowship. 


\begin{thebibliography}{0}

\bibitem[CdV18]{C}  Yves Colin de Verdi\`ere, 
	\emph{Spectral theory of pseudo-differential operators of degree 0 and application to forced linear waves,\/}
	preprint, \arXiv{1804.03367}.

\bibitem[CS18]{SC} Yves Colin de Verdi\`ere and Laure Saint-Raymond, 
	\emph{Attractors for two dimensional waves with homogeneous Hamiltonians of degree 0,\/}
	to appear in Comm. Pure Appl. Math., \arXiv{1801.05582}.

\bibitem[DaDy13]{DaD} Kiril Datchev and Semyon Dyatlov,
	\emph{Fractal Weyl laws for asymptotically hyperbolic manifolds,\/}
	Geom. Funct. Anal. \textbf{23}(2013), 1145--1206.   

\bibitem[Dy12]{dy} Semyon Dyatlov,
	\emph{Asymptotic distribution of quasi-normal modes for Kerr--de Sitter black holes,\/}
	Ann. Henri Poincar\'e \textbf{13}(2012), 1101--1166.

\bibitem[DyGu16]{DG}  Semyon Dyatlov and Colin Guillarmou,
        \emph{Pollicott--Ruelle resonances for open systems,\/}
        Ann. Inst. Henri Poincar\'e (A), \textbf{17}(2016), 3089--3146.

\bibitem[DyZw16]{DZ} Semyon Dyatlov and Maciej Zworski, 
	\emph{Dynamical zeta functions for Anosov flows via microlocal analysis,\/}
	Ann. Sci. Ec. Norm. Sup\'er. \textbf{49}(2016), 543--577.

\bibitem[DyZw17]{zazi} Semyon Dyatlov and Maciej Zworski,
	\emph{Ruelle zeta function at zero for surfaces,\/}
	Inv.~Math. \textbf{210}(2017), 211--229.

\bibitem[DyZw]{res} Semyon Dyatlov and Maciej Zworski, 
	\emph{Mathematical theory of scattering resonances,\/}
                book in preparation; \url{http://math.mit.edu/~dyatlov/res/}

\bibitem[HaVa15]{hb} Nick Haber and Andr\'as Vasy, 
	\emph{Propagation of singularities around a Lagrangian submanifold of radial points,\/}
	Bull.~Soc.~Math.~France \textbf{143}(2015), 679--726. 

\bibitem[HMV04]{hmv} Andrew Hassell, Richard Melrose, and Andr\'as Vasy,
	\emph{Spectral and scattering theory for symbolic potentials of order zero,\/}
	Adv. Math. \textbf{181}(2004), 1--87.

\bibitem[HiVa16]{HiV} Peter Hintz and Andr\'as Vasy,
	\emph{The global non-linear stability of the Kerr-de Sitter family of black holes,\/}
	to appear in Acta Math., \arXiv{1606.04014}. 

\bibitem[H\"oIII]{H3} Lars H{\"o}rmander,
       \emph{The Analysis of Linear Partial Differential Operators III. Pseudo-Differential Operators,\/}
        Springer Verlag, 1985.

\bibitem[H\"oIV]{H4} Lars H{\"o}rmander,
       \emph{The Analysis of Linear Partial Differential Operators IV. Pseudo-Differential Operators,\/}
        Springer Verlag, 1985.

\bibitem[JMP84]{jemp}  Arne Jensen, \'Eric Mourre, and Peter Perry,
	\emph{Multiple commutator estimates and resolvent smoothness in quantum scattering theory,\/}
	Ann. Inst. H. Poincar\'e Phys. Th\'eor. \textbf{41}(1984), 207--225.

\bibitem[Me94]{mel} Richard B. Melrose,
        \emph{Spectral and scattering theory for the Laplacian on asymptotically Euclidian spaces,\/}
        in \emph{Spectral and scattering theory} (M. Ikawa, ed.), 
        Marcel Dekker, 1994 

\bibitem[NiZh99]{Surf} Igor Nikolaev and Evgeny Zhuzhoma,
	\emph{Flows on 2-dimensional Manifolds. An Overview,\/}
	Springer, 1999.

\bibitem[Ra73]{Ra73} James Ralston,
	\emph{On stationary modes in inviscid rotating fluid,\/}
	J. Math. Anal. Appl. \textbf{44}(1973), 366--383.

\bibitem[Va13]{Va} Andr\'as Vasy,
        \emph{Microlocal analysis of asymptotically hyperbolic and Kerr--de Sitter spaces,\/}
        with an appendix by Semyon Dyatlov,
        Invent. Math. \textbf{194}(2013), 381--513. 

\bibitem[Zw16]{V4D} Maciej Zworski,
	\emph{Resonances for asymptotically hyperbolic manifolds: Vasy's method revisited,\/}
	J.~Spectr.~Theory. \textbf{6}(2016), 1087--1114.


        

\end{thebibliography}
\end{document}